\documentclass[a4paper,11pt]{article}
\usepackage{amsmath,amsthm,amssymb}
\usepackage[mathscr]{eucal}
\usepackage[bookmarks=false]{hyperref}

\numberwithin{equation}{section}

{\theoremstyle{definition}\newtheorem{definition}{Definition}[section]

\newtheorem{remark}[definition]{Remark}
\newtheorem{example}[definition]{Example}}

\newtheorem{lemma}[definition]{Lemma}
\newtheorem{theorem}[definition]{Theorem}
\newtheorem{corollary}[definition]{Corollary}

\newcommand{\M}{\operatorname{M}}
\newcommand{\C}{\mathbb{C}}
\newcommand{\notembed}[1]{\not\prec_{#1}}
\newcommand{\embed}[1]{\prec_{#1}}
\newcommand{\F}{\mathbb{F}}
\newcommand{\cR}{\mathcal{R}}

\newcommand{\actson}{\curvearrowright}
\newcommand{\SL}{\operatorname{SL}}
\newcommand{\rL}{\operatorname{L}}
\newcommand{\Aut}{\operatorname{Aut}}

\newcommand{\N}{\mathbb{N}}
\newcommand{\T}{\mathbb{T}}
\newcommand{\Z}{\mathbb{Z}}
\newcommand{\cF}{\mathcal{F}}

\newcommand{\cV}{\mathcal{V}}
\newcommand{\id}{\mathord{\operatorname{id}}}
\newcommand{\si}{\sigma}

\newcommand{\recht}{\rightarrow}
\newcommand{\cU}{\mathcal{U}}
\newcommand{\vphi}{\varphi}

\newcommand{\R}{\mathbb{R}}
\newcommand{\al}{\alpha}
\newcommand{\eps}{\varepsilon}

\newcommand{\Tr}{\operatorname{Tr}}
\newcommand{\ovt}{\overline{\otimes}}
\newcommand{\B}{\operatorname{B}}
\newcommand{\om}{\omega}
\newcommand{\cP}{\mathcal{P}}
\newcommand{\cZ}{\mathcal{Z}}

\newcommand{\Ker}{\operatorname{Ker}}
\newcommand{\cK}{\mathcal{K}}

\newcommand{\cH}{\mathcal{H}}
\newcommand{\cJ}{\mathcal{J}}
\newcommand{\ot}{\otimes}
\newcommand{\cL}{\mathcal{L}}
\newcommand{\Lambdatil}{\widetilde{\Lambda}}
\newcommand{\dis}{\displaystyle}
\newcommand{\Ad}{\operatorname{Ad}}
\newcommand{\cG}{\mathcal{G}}
\newcommand{\cM}{\mathcal{M}}
\newcommand{\dpr}{^{\prime\prime}}
\newcommand{\be}{\beta}

\newcommand{\vphitil}{\widetilde{\vphi}}
\newcommand{\rB}{\operatorname{B}}
\newcommand{\m}{\mathord{\text{\rm m}}}
\newcommand{\lspan}{\operatorname{span}}
\newcommand{\Cal}{\mathcal}
\newcommand{\Mtil}{\widetilde{M}}
\newcommand{\Ntil}{\widetilde{N}}
\newcommand{\D}{\operatorname{D}}
\newcommand{\Stab}{\operatorname{Stab}}
\newcommand{\rO}{\operatorname{O}}
\newcommand{\PSL}{\operatorname{PSL}}
\newcommand{\Mbol}{M^\circ}
\newcommand{\rZ}{\operatorname{Z}}
\newcommand{\Ufin}{\mathcal{U}_{\rm fin}}
\newcommand{\binor}{\otimes_{\rm binor}}
\newcommand{\Nop}{N^{\rm op}}
\newcommand{\cN}{\mathcal{N}}
\newcommand{\mybim}[2]{$\bigl(#1\bigr)$--$\bigl(#2\bigr)$--bimodule}
\newcommand{\Ytil}{\widetilde{Y}}
\newcommand{\cS}{\mathcal{S}}

\newcommand{\Char}{\operatorname{Char}}
\newcommand{\tpr}{^{\prime\prime\prime}}
\newcommand{\Om}{\Omega}

\newcommand{\bim}[3]{\mathord{\raisebox{-0.4ex}[0ex][0ex]{\scriptsize $#1$}{#2}\hspace{-0.2ex}\raisebox{-0.4ex}[0ex][0ex]{\scriptsize $#3$}}}

\newcommand{\cb}[1]{\| #1 \|_{\text{\rm cb}}}

\begin{document}

\begin{center}
{\LARGE\bf Group measure space decomposition of II$_1$ factors
 \vspace{0.5ex}\\ and W*-superrigidity}

\bigskip

{\sc by Sorin Popa\footnote{Partially supported by NSF Grant
DMS-0601082}\footnote{Mathematics Department; University of
    California at Los Angeles, CA 90095-1555 (United States).
    \\ E-mail: popa@math.ucla.edu} and Stefaan Vaes\footnote{Partially
    supported by ERC Starting Grant VNALG-200749, Research
    Programme G.0231.07 of the Research Foundation --
    Flanders (FWO) and K.U.Leuven BOF research grant OT/08/032.}\footnote{Department of Mathematics;
    K.U.Leuven; Celestijnenlaan 200B; B--3001 Leuven (Belgium).
    \\ E-mail: stefaan.vaes@wis.kuleuven.be}}
\end{center}

\begin{abstract}\noindent
We prove a ``unique crossed product decomposition'' result for group
measure space II$_1$ factors $\rL^\infty(X)\rtimes \Gamma$ arising
from arbitrary free ergodic probability measure preserving (p.m.p.)
actions of groups $\Gamma$ in a fairly large family $\cG$, which
contains all free products of a Kazhdan group and a non-trivial
group, as well as certain amalgamated free products over an amenable
subgroup. We deduce that if $T_n$ denotes the group of upper
triangular matrices in $\PSL(n,\Z)$, then any free, mixing p.m.p.
action of $\Gamma=\PSL(n,\Z)*_{T_n} \PSL(n,\Z)$ is W$^*$-superrigid,
i.e. any isomorphism between $\rL^\infty(X)\rtimes \Gamma$ and an
arbitrary group measure space factor $\rL^\infty(Y)\rtimes \Lambda$,
comes from a conjugacy of the actions. We also prove that for many
groups $\Gamma$ in the family $\cG$, the Bernoulli actions of
$\Gamma$ are W$^*$-superrigid.
\end{abstract}

\section{Introduction} \label{sec.intro}

Rigidity results have by now appeared in many areas of mathematics,
and in several forms. The most frequently encountered is when two
mathematical objects with rich structure that are known to be
equivalent in some ``weak sense'', which ignores part of the
structure, are shown to be isomorphic as objects with the full
structure. In the best of cases, such a result will also show that
morphisms which are equivalences in the weak sense are equivalent to
isomorphisms in the stronger category, thus leading to complete
classification results and calculation of invariants.

Von Neumann algebras (also called W$^*$-algebras) provide a most
natural framework for rigidity. In fact, such phenomena are at the
very core of this subject, relating it at the outset with group
theory and ergodic theory. This is due to the Murray and von Neumann
classical {\it group measure space construction}, which associates
to a free ergodic measure preserving action $\Gamma \actson X$, of a
countable group $\Gamma$ on a probability space $(X,\mu)$, a von
Neumann algebra (called II$_1$ {\it factor}) $\rL^\infty(X)\rtimes
\Gamma$, through a crossed product type construction \cite{MvN1}. The
study of these objects in terms of their ``initial data'' $\Gamma
\actson X$ has been a central theme of the subject since the early
1940s. It soon led to a new area in ergodic theory, studying group
actions up to {\it orbit equivalence} ({\it OE}), i.e.\ up to
isomorphism of probability spaces carrying the orbits of actions
onto each other, since an OE of actions $\Gamma \curvearrowright X$,
$\Lambda \curvearrowright Y$ has been shown in \cite{Si} to be ``the same
as'' an algebra isomorphism $\rL^\infty(X)\rtimes \Gamma \simeq
\rL^\infty(Y)\rtimes \Lambda$ taking the {\it group measure space Cartan subalgebras}
$\rL^\infty(X)$, $\rL^\infty(Y)$ onto each other.

Thus, {\it W$^*$-equivalence} (or {\it von Neumann equivalence}) of group actions, requiring
isomorphism of their group measure space algebras, is weaker than
OE. Since there are examples of non-OE actions whose group measure
space factors are isomorphic \cite{CJ,OP2}, it is in general
strictly weaker. In turn, it has been known since \cite{MvN43,D59}
that OE is much weaker than classical {\it conjugacy} (or {\it
isomorphism}), which for free actions $\Gamma \curvearrowright X$,
$\Lambda \curvearrowright Y$ requires isomorphism of probability
spaces $\Delta: (X,\mu) \simeq (Y,\nu)$ satisfying $\Delta \Gamma
\Delta^{-1}=\Lambda$ (so in particular $\Gamma\simeq \Lambda$).
Rigidity in this context occurs whenever one can establish that
W$^*$- or OE-equivalence of certain group actions $\Gamma
\curvearrowright X$, $\Lambda \curvearrowright Y$, forces the
groups, or the actions, to share some common properties. The ideal
such result, labeled W$^*$- (respectively OE-) {\it superrigidity},
recovers the isomorphism class of $\Gamma \curvearrowright X$, from
its W$^*$-class (resp.\ OE-class).

W$^*$- and OE-rigidity can only occur for non-amenable groups, since
by classical results of Connes \cite{C76}, all II$_1$ factors
$\rL^\infty(X)\rtimes \Gamma$ with $\Gamma$ amenable are mutually
isomorphic and by \cite{OW,CFW} they are undistinguishable under OE
as well. But the non-amenable case is extremely complex and although
signs of rigidity where detected early on \cite{MvN43,D63,Mc70,C75},
for many years progress has been slow, despite several
breakthrough discoveries in the 1980s \cite{C80,Zi80,CJ85,
CowHaag}. This changed dramatically over the last decade, with the
advent of a variety of striking rigidity results, in both group
measure space II$_1$ factors and OE ergodic theory:
\cite{Fu1,Fu2,G2,G1,PBetti,MoSh,Hj,
Oz,P1,P2,P0,HjKe,IPP,P-gap,PV1,Pet,kida1,
V-bim,OP1,Ioana,OP2,PV2,PV4,PV3,C-H,kida-amal}.

Our purpose in this paper is to investigate the most ``extreme'' of
the W$^*$-rigidity phenomena mentioned above, i.e.\
W$^*$-superrigidity. Thus, we seek to find classes of group actions
$\Gamma \actson X$ with the property that any isomorphism between
$\rL^\infty(X)\rtimes \Gamma$ and any other group measure space
factor $\rL^\infty(Y)\rtimes \Lambda$, arising from an arbitrary
free ergodic p.m.p.\  action $\Lambda \actson Y$, comes from a
conjugacy of the actions $\Gamma\actson X$, $\Lambda \actson Y$.

Note that W$^*$-superrigidity for an action $\Gamma \actson X$ is
equivalent to the ``sum'' between its OE-superrigidity and the
uniqueness, up to unitary conjugacy, of $\rL^\infty(X)$ as a group
measure space Cartan subalgebra in $\rL^\infty(X)\rtimes \Gamma$.
This makes W$^*$-superrigidity results extremely difficult to
obtain, since each one of these problems is notoriously hard. But
while several large families of OE-superrigid actions have been
discovered over the last ten years
\cite{Fu2,P0,P-gap,kida1,Ioana,kida-amal}, unique Cartan
decomposition proved to be much more challenging to establish, and
the only existing results cover very particular group actions. Thus,
a first such result, obtained by Ozawa and the first named author in
\cite{OP1}, shows that given any profinite action $\Gamma \actson
X$, of a product of free groups $\Gamma = \F_{n_1}\times \cdots
\times \F_{n_k}$, with $k\geq 1$, $2\leq n_i \leq \infty$, any
Cartan subalgebra of $M=\rL^\infty(X) \rtimes \Gamma$ (i.e.\ any
maximal abelian subalgebra whose normalizer generates $M$), is
unitary conjugate to $\rL^\infty(X)$. A similar result, covering a
more general class of groups $\Gamma$, was then proved in
\cite{OP2}. More recently, Peterson showed in \cite{jpete} that
factors arising from profinite actions of non-trivial free products
$\Gamma=\Gamma_1
* \Gamma_2$, with at least one of the $\Gamma_i$ not having the
Haagerup property, have unique group measure space Cartan
subalgebra, up to unitary conjugacy. But so far, none of these group
actions could be shown to be OE-superrigid. Nevertheless, an
intricate combination of results in \cite{Ioana,OP2,jpete} were used
to prove the existence of virtually W$^*$-superrigid group actions
$\Gamma \actson X$ in \cite{jpete}, by a Baire category argument
(following \cite{Fu1}, {\it virtual} means that the ensuing
conjugacy of $\Gamma \actson X$ and the target actions $\Lambda
\actson Y$ is up to finite index subgroups of $\Gamma, \Lambda$).

In this paper, we establish a very general unique Cartan
decomposition result, which allows us to obtain a wide range of
W$^*$-superrigid group actions. Thus, we first prove the uniqueness,
up to unitary conjugacy, of the group measure space Cartan
subalgebra in the II$_1$ factor given by an \emph{arbitrary} free
ergodic p.m.p.\ action of any group $\Gamma$ belonging to a large
family $\cG$ of amalgamated free product groups. By combining this
with Kida's OE-superrigidity in \cite{kida-amal}, we deduce
that any free, mixing p.m.p.\ action of $\Gamma=\PSL(n,\Z)*_{T_n}
\PSL(n,\Z)$ is W$^*$-superrigid. In combination with
\cite{P0,P-gap}, we prove that for many groups $\Gamma$ in the
family $\cG$, the Bernoulli actions of $\Gamma$ are
W$^*$-superrigid. In combination with Gaboriau's work \cite{G2} on cost,
we find new groups $\Gamma$ for which all group measure space
II$_1$ factors $\rL^\infty(X) \rtimes \Gamma$ have trivial fundamental group.

\subsection{Statements of main results}

More precisely, our family $\cG$ contains all non-trivial free
products $\Gamma = \Gamma_1 * \Gamma_2$ with $\Gamma_1$ satisfying
one of the following rigidity properties: $\Gamma_1$ contains a
non-amenable subgroup with the relative property (T) or $\Gamma_1$
contains two non-amenable commuting subgroups. The family $\cG$ also
contains certain amalgamated free products $\Gamma_1 *_\Sigma
\Gamma_2$ over amenable subgroups $\Sigma$, see Definition
\ref{def.G}.

Our results can be summarized as follows.

\begin{theorem}[See Theorem \ref{thm.uniqueCartan}] \label{thm.uniqueCartan-intro}
Let $\Gamma$ be a group in the family $\cG$ and $\Gamma \actson (X,\mu)$
an arbitrary free ergodic p.m.p.\ action. Denote $M = \rL^\infty(X) \rtimes \Gamma$.
Whenever $\Lambda \actson (Y,\eta)$ is a free ergodic p.m.p.\ action such
that $M = \rL^\infty(Y) \rtimes \Lambda$, there exists a unitary $u \in M$
such that $\rL^\infty(Y) = u \rL^\infty(X) u^*$.
\end{theorem}

We mention that the most general version of the above theorem (see
Theorem \ref{thm.uniqueCartan}) allows to handle amplifications of
the group measure space factors $M$ as well.

\begin{theorem}[See Theorem \ref{thm.kida}]\label{thm.kida-intro}
Let $n \geq 3$ and denote by $T_n$ the subgroup of upper triangular matrices
in $\PSL(n,\Z)$. Put $\Gamma = \PSL(n,\Z) *_{T_n} \PSL(n,\Z)$. Then,
every free p.m.p.\ mixing action of $\Gamma$ is W$^*$-superrigid.
\end{theorem}

Whenever $\Gamma$ is an infinite group and $(X_0,\mu_0)$ a non-trivial
probability space, denote by $\Gamma \actson (X_0,\mu_0)^\Gamma$ the
\emph{Bernoulli action} of $\Gamma$ with base space $(X_0,\mu_0)$,
given by $(g \cdot x)_h = x_{g^{-1}h}$ for all $g,h \in \Gamma$ and $x \in X_0^\Gamma$.

\begin{theorem}[See Theorem \ref{thm.stable-superrigid}]\label{thm.superrigid-Bernoulli}
The Bernoulli action $\Gamma \actson (X_0,\mu_0)^\Gamma$ of all of the following
groups is W$^*$-superrigid.
\begin{itemize}
\item $\Gamma = \Gamma_1 *_\Sigma \Gamma_2$ with the following assumptions:
$\Gamma_1$ has property (T), $\Sigma$ is an infinite, amenable, proper normal
subgroup of $\Gamma_2$ and there exist $g_1,\ldots,g_k \in \Gamma_1$ such
that $\bigcap_{i=1}^k g_i \Sigma g_i^{-1} = \{e\}$.

For instance, we can take $\Gamma = \PSL(n,\Z) *_\Sigma (\Sigma \times \Lambda)$,
where $\Sigma < T_n$ is an infinite subgroup of the upper triangular matrices and
$\Lambda$ is an arbitrary non-trivial group.
\item $\Gamma = (H \times H) *_\Sigma \Gamma_2$ where $H$ is a finitely generated
non-amenable group with trivial center, $\Sigma$ is an infinite amenable subgroup
of $H$ that we embed diagonally in $H \times H$ and $\Sigma$ is a proper normal subgroup of $\Gamma_2$.
\end{itemize}
\end{theorem}

We in fact obtain a more general version of this result as Theorem
\ref{thm.stable-superrigid}, which covers generalized Bernoulli
actions, Gaussian actions and certain co-induced actions (see
Examples \ref{ex.firstex}, \ref{ex.secondex} and \ref{ex.Haagerup}).

Our methods also provide the following new examples of II$_1$
factors which cannot be written as group measure space factors.

\begin{theorem} \label{thm.no-Cartan}
Let $\Gamma \in \cG$ and assume that $\Gamma$ is ICC. Let $\Gamma \actson (X,\mu)$
be an arbitrary ergodic p.m.p.\ action and put $M = \rL^\infty(X) \rtimes \Gamma$.
Then, $M$ is a II$_1$ factor. If, for some $t > 0$, the II$_1$ factor $M^t$ admits
a group measure space decomposition, then the action $\Gamma \actson (X,\mu)$ must be free.
Thus, $\rL(\Gamma)$ and all the factors of the form
$\rL^\infty(X)\rtimes \Gamma$ corresponding to non-free actions
$\Gamma \actson X$, do not admit a group measure space
decomposition.
\end{theorem}

\subsection{Rigidity of bimodules}

Another typical rigidity paradigm encountered in mathematics is when
certain invariants of mathematical objects which are supposed to
take values in a certain range, are shown to take values in a much
smaller subset. Group measure space II$_1$ factors provide a natural
framework for this type of rigidity as well, due to Murray and von
Neumann's {\it continuous dimension} and a related invariant for
II$_1$ factors $M$: the {\it fundamental group} $\cF(M)$. This is
defined as the set of ratios $\tau(p)/\tau(q)\in \R_+$, over all
projections $p,q\in M$ with $pMp\simeq qMq$, where $\tau$ denotes
the (unique) normalized trace (=dimension function) on $M$.
Equivalently, $\cF(M)=\{t>0 \mid M^t\simeq M\}$. Thus, since the
range of the dimension function is all $[0,1]$, the group $\cF(M)$
seems to always be equal to $\R_+$. Supporting evidence comes from
the case $M=\rL^\infty(X)\rtimes \Gamma$ with $\Gamma$ amenable,
when this is indeed the case (cf.\ \cite{MvN43}). So it came as a
striking surprise when Connes showed that all factors
$\rL^\infty(X)\rtimes \Gamma$ with $\Gamma$ an ICC Kazhdan group and
$\Gamma \actson X$ ergodic, have countable fundamental group (\cite{C80}).

The important progress in W$^*$-rigidity in recent years, led to the
first actual computations of fundamental groups $\cF(M)$ of group
measure space factors $M=\rL^\infty(X)\rtimes \Gamma$: from the first
such examples in \cite{PBetti}, where $\cF(M)=1$, to examples of
factors $M$ with $\cF(M)$ any prescribed countable subgroup of $\R_+$ \cite{P1}
(see also \cite{IPP,Houd}), and most recently examples
with $\cF(M)$ uncountable, yet different from $\R_+$ \cite{PV2,PV3}.
We mention in this respect that all Bernoulli actions $\Gamma
\actson (X,\mu)$ appearing in Theorem \ref{thm.superrigid-Bernoulli}
give rise to II$_1$ factors $\rL^\infty(X) \rtimes \Gamma$ with
trivial fundamental group (see Remark \ref{rem.fundamentalgroup}).
In the same spirit, Corollary \ref{cor.Sfactor} provides new examples
of groups $\Gamma$ such that $\rL^\infty(X) \rtimes \Gamma$ has trivial
fundamental group for all free ergodic p.m.p.\ actions $\Gamma \actson (X,\mu)$.

Another occurrence of the same type of rigidity paradigm is related
to Jones' index for subfactors, a numerical invariant for inclusions
of II$_1$ factors $N \subset M$ which, like the fundamental group,
is defined with the help of the Murray-von Neumann continuous
dimension. A priori, the range of the index could well be all
$\R_+$, but in his seminal work \cite{Jo83}, Jones proved that it is
subject to very surprising restrictions.

A unique feature of the II$_1$ factor framework is that it allows a
unifying approach to the two types of rigidity phenomena
(W$^*$-rigidity and restrictions on invariants), by considering
finite index bimodules between factors (as a generalization of
isomorphism between factors and their amplifications). Thus, by
explicitly calculating all bimodules between factors in a certain
class, one also obtains the fundamental group of the corresponding
factors, as well as all possible indices of its subfactors.

In the last section \ref{sec.bimodules}, we combine a generalization
of Theorem \ref{thm.uniqueCartan-intro} with the cocycle
superrigidity theorems of \cite{P0,P-gap} and techniques from
\cite{V-bim}, to give examples of group actions $\Gamma \actson
(X,\mu)$ such that the mere existence of a finite index bimodule
between $\rL^\infty(X) \rtimes \Gamma$ and $\rL^\infty(Y)\rtimes
\Lambda$ for an arbitrary free ergodic p.m.p.\ action $\Lambda
\actson (Y,\eta)$, implies that the groups $\Gamma,\Lambda$ are
virtually isomorphic and their actions $\Gamma \actson X$, $\Lambda
\actson Y$ are virtually conjugate in a very precise sense (see
Theorem \ref{thm.bimodules} and Example \ref{ex.bimodules}). In
particular, the fundamental group of any of these II$_1$ factors
$\rL^\infty(X)\rtimes \Gamma$ is trivial and the index of all their
subfactors is an integer.

\subsection{Comments on the proofs}

As we mentioned before, the main difficulty in obtaining
W$^*$-superrigidity lies in proving the uniqueness of the group measure
space Cartan decomposition. Indeed, because once such a result is
established, W$^*$-superrigidity can be derived from existing
OE-superrigidity results. In our case, \ref{thm.kida-intro} and
\ref{thm.superrigid-Bernoulli} will follow from our uniqueness of the group
measure space Cartan subalgebra in Theorem
\ref{thm.uniqueCartan-intro} and the OE superrigidity theorems in
\cite{P0,P-gap,kida-amal}. By using {\it intertwining subalgebras}
techniques (\cite{PBetti, P1}), in order to prove the uniqueness, up
to unitary conjugacy, of the Cartan subalgebra $A=\rL^\infty(X)$ of a
group measure space factor $M=A \rtimes \Gamma$, it is
sufficient to prove that given any other group measure space
decomposition $M=B \rtimes \Lambda$, $B = \rL^\infty(Y)$, there exists a $B$-$A$-bimodule
$\Cal H\subset \rL^2(M)$, which is finitely generated
over $A$, a property that we denote by $B \prec_M A$. To get such bimodules, we use the
deformation-rigidity theory introduced in
\cite{PBetti, P1, P2} (see \cite{P-ICM} for a survey), and in fact
the whole array of subsequent developments in \cite{IPP, PV1, P-gap,
V-bim, Houd, C-H}, etc. But in order to ``locate'' the position of the {\it target} Cartan subalgebra
$B$, with respect to the initial ({\it source})
Cartan subalgebra $A$, through these techniques, one needs some
amount of rigidity for one of the group actions and a deformation property
for the other
(like for example in 6.2 of \cite{PBetti}, 7.1 of \cite{P2}, 7.7 of
\cite{IPP}, 1.5 of \cite{P-gap}, etc).

Since in W$^*$-superrigidity statements all assumptions must be on the
side of the source Cartan subalgebra $A=\rL^\infty(X)$, it is thus crucial
to show that either the deformation or the rigidity
properties of $\Gamma \actson X$ {\it automatically transfer} to
$\Lambda \actson Y$. It is precisely the lack of satisfactory
``transfer'' results that so far prevented from obtaining
W$^*$-superrigidity results. We solve this problem here
by proving in Section 2 some very general ``transfer of rigidity'',
from the source to the target side. While the proofs of these results
are quite subtle, let us give here a heuristic explanation.

Assume that $A \rtimes \Gamma = M = B \rtimes \Lambda$ and denote by $(u_g)_{g \in \Gamma}$, resp.\ $(v_s)_{s \in \Lambda}$, the canonical unitaries in $A \rtimes \Gamma$, resp.\ $B \rtimes \Lambda$. Every element $x \in M$, has a Fourier expansion $x = \sum_{g \in \Gamma} x_g u_g$ with $x_g \in A$ and we call the $x_g$ the Fourier coefficients of $x$ w.r.t.\ $\{u_g\}$. We similarly define Fourier coefficients w.r.t.\ $\{v_s\}$. We assume that $\Gamma_1 < \Gamma$ is a non-amenable subgroup with the relative property (T) and try to transfer this rigidity property to some rigidity for $\Lambda$. More precisely, we show that for any deformation
$\phi_n$ of $M$ (i.e.\ a sequence of c.p.\ maps on $M$ tending
pointwise in the Hilbert norm to $\id_M$), there exist a large $n$
and an infinite subset $\{s_k\}_k\subset \Lambda$, such that
$\phi_n(v_{s_k}) \approx v_{s_k}$, $\forall k$, and such that every Fourier
coefficient of $v_{s_k}$ w.r.t.\ $\{u_g\}$ tends to zero in $\| \,\cdot\, \|_2$ as $k \recht \infty$.
We construct as follows the set $\{s_k\}_k$. The functions
$\psi_n(s) = \tau(\phi_n(v_s)v_s^*)$ are positive definite and hence,
define the c.p.\ maps $\Psi_n$ by $\Psi_n(\sum_s b_s v_s) = \sum_s \psi_n(s)
b_s v_s$. Since $\Psi_n \rightarrow \id_M$, the relative property (T) of
$\Gamma_1 < \Gamma$ ensures that $\Psi_n(u_g) \approx u_g$ uniformly
in $g \in \Gamma_1$. This forces $\psi_n(s) \approx 1$ for many of the $s \in \Lambda$
in the support of the Fourier expansion of $u_g, g \in \Gamma_1$ w.r.t.\ $\{v_s\}$. Among the $s \in \Lambda$ with
$\psi_n(s) \approx 1$, we can find a sequence $s_k$ such that the Fourier
coefficients of $v_{s_k}$ tend to zero as $k \recht \infty$, because otherwise, it will follow that the
$u_g, g \in \Gamma_1$ can roughly be intertwined into $A$, contradicting the
non-amenability of $\Gamma_1$.

As it turns out, if we assume that $\Gamma_1$ is freely complemented in
$\Gamma$, i.e.\ $\Gamma=\Gamma_1 * \Gamma_2$, then the ``tiny''
initial information about the group $\Lambda$ provided by the transfer of rigidity, is enough to prove that $B \prec_M A$. To see this, we first notice that if we apply the above transfer result to the word length deformation $\m_\rho(\sum_g
a_g u_g) = \sum_g \rho^{|g|} a_g u_g$, as $\rho \rightarrow
1$, then, for $\rho$ close enough to $1$, we have $\m_\rho(v_{s_k}) \approx v_{s_k}$ uniformly in $k$.
This implies that in the Fourier decomposition with respect to
$\{u_g\}$, all $v_{s_k}$ are almost supported by words $g\in
\Gamma=\Gamma_1 * \Gamma_2$ of length uniformly bounded by some $K$. On the other hand, since the Fourier coefficients of $v_{s_k}$ w.r.t.\ $\{u_g\}$ tend to $0$ in $\| \, \cdot \, \|_2$, the support of the Fourier expansion of the $v_{s_k}$ lies, as $k \recht \infty$, essentially outside any given finite subset of $\Gamma$. If, by contradiction $B \not\prec_M A$, results from \cite{P1,IPP} provide a unitary $w \in B$ such that the Fourier expansion of $w$ is essentially supported by a set of words $g \in \Gamma$ of length much larger than $K$ ($|g| \geq 2 K$
will do). As we will explain now, this implies that $v_{s_k} w v_{s_k}^*w^*$ and
$w^*v_{s_k}wv_{s_k}^*$ are almost orthogonal, contradicting the abelianess of $B$.

Indeed, first assume for simplicity that all $v_{s_k}$ lie in the span of $Au_g$, $|g|\leq K$, with the support of the Fourier expansion of $v_{s_k}$ w.r.t.\ $\{u_g\}$ tending to infinity in $\Gamma$. Similarly, assume that $w$ exactly lies in the span of $Au_g$, $|g| \geq 2K$. Then, one concludes that all $g \in \Gamma$ in the support of $w^*v_{s_k}wv_{s_k}^*$ eventually have their first $K$ letters in a fixed finite set independent of $k$, while all $g \in \Gamma$ in the support of $v_{s_k} w v_{s_k}^*w^*$, have their first $K$ letters eventually (as $k \recht \infty$) outside any fixed finite set. In reality, we can only approximate in $\|\,\cdot\,\|_2$ and uniformly in $k$, the unitaries $v_{s_k}$ by elements $v_{s_k}'$ with such good properties. We similarly approximate $w$ by $w'$. But since our reasoning involves products of 4 elements, the Hilbert norm estimates cannot be handled unless one
can control the uniform norms of $w-w'$, $v_{s_k}-v_{s_k}'$. We handle this problem through repeated ``trimming'' of elements, via
Herz-Schur multiplier techniques.

\section{Preliminaries}\label{sec.prelim}

If $\Gamma \actson (A,\tau)$ is a trace preserving action of a countable group $\Gamma$, we denote by $A \rtimes \Gamma$ the \emph{crossed product} von Neumann algebra, which is the unique tracial von Neumann algebra generated by $A$ and the group of unitaries $(u_g)_{g \in \Gamma}$ satisfying
$$u_g a u_g^* = \si_g(a) \quad\text{for all}\;\; g \in \Gamma , a \in A \quad\text{and}\quad \tau(a u_g) = \begin{cases} \tau(a) &\;\;\text{if}\; g = e \; , \\ 0 &\;\;\text{if}\; g \neq e \; .\end{cases}$$
When $\Gamma \actson (X,\mu)$ is a free ergodic p.m.p.\ action, the crossed product $\rL^\infty(X) \rtimes \Gamma$ is called the \emph{group measure space} II$_1$ factor associated with $\Gamma \actson (X,\mu)$. Then, $\rL^\infty(X)$ is a \emph{Cartan subalgebra} of $\rL^\infty(X) \rtimes \Gamma$, called a group measure space Cartan subalgebra.

We denote by $\rL(\Gamma)$ the group von Neumann algebra of a countable group $\Gamma$.

Recall that two free ergodic p.m.p.\ actions $\Gamma \actson (X,\mu)$ and $\Lambda \actson (Y,\eta)$ are called
\begin{itemize}
\item \emph{conjugate}, if there exists an isomorphism $\Delta : X \recht Y$ of probability spaces and an isomorphism $\delta : \Gamma \recht \Lambda$ of groups such that $\Delta(g \cdot x) = \delta(g) \cdot \Delta(x)$ almost everywhere,
\item \emph{orbit equivalent}, if there exists an isomorphism $\Delta : X \recht Y$ of probability spaces such that $\Delta(\Gamma \cdot x) = \Lambda \cdot \Delta(x)$ for almost all $x \in X$,
\item \emph{W$^*$-equivalent} (or von Neumann equivalent), if $\rL^\infty(X) \rtimes \Gamma \cong \rL^\infty(Y) \rtimes \Lambda$.
\end{itemize}

If two actions are conjugate, they are obviously orbit equivalent. On the other hand, two actions are orbit equivalent if and only if there exists an isomorphism $\rL^\infty(X) \rtimes \Gamma \cong \rL^\infty(Y) \rtimes \Lambda$ sending $\rL^\infty(X)$ onto $\rL^\infty(Y)$, see \cite{Si,FM}.

\subsection{Bimodules and weak containment}

Let $M,N$ be tracial von Neumann algebras. An \emph{$M$-$N$-bimodule $\bim{M}{\cH}{N}$} is a Hilbert space $\cH$ equipped with a normal representation $\pi$ of $M$ and a normal anti-representation $\pi'$ of $N$ such that $\pi(M)$ and $\pi'(N)$ commute.

Given the bimodules $\bim{M}{\cH}{N}$ and $\bim{N}{\cK}{P}$, one can define the \emph{Connes tensor product} $\cH \ot_N \cK$ which is an $M$-$P$-bimodule, see \cite[V.Appendix B]{connes}.

An $M$-$N$-bimodule can be seen as well as a representation of the C$^*$-algebra $M \binor \Nop$. If $\bim{M}{\cH^1}{N}$ and $\bim{M}{\cH^2}{N}$ are $M$-$N$-bimodules, we say that $\cH^1$ is \emph{weakly contained} in $\cH^2$ if the corresponding representations $\pi^1, \pi^2$ of $M \binor \Nop$ satisfy $\Ker \pi^1 \supset \Ker \pi^2$. Weak containment behaves well with respect to the Connes tensor product: if $\bim{M}{\cH^1}{N}$ is weakly contained in $\bim{M}{\cH^2}{N}$, then $\cK \ot_M \cH^1$ is weakly contained in $\cK \ot_M \cH^2$ for every $P$-$M$-bimodule $\cK$ (see e.g.\ \cite[Lemma 1.7]{claire}). A similar statement holds for tensor products on the right.

We call $\bim{M}{\rL^2(M)}{M}$ the \emph{trivial $M$-$M$-bimodule} and define the \emph{coarse $M$-$M$-bimodule} as the Hilbert space $\rL^2(M) \ot \rL^2(M)$ equipped with the bimodule structure $a \cdot \xi \cdot b = (a \ot 1) \xi (1 \ot b)$. A finite von Neumann algebra $M$ is \emph{injective} if the trivial $M$-$M$-bimodule is weakly contained in the coarse $M$-$M$-bimodule.

The first type of bimodule that we encounter in this article, is the following. Let $\Gamma \actson (Q,\tau)$ be a trace preserving action and put $M = Q \rtimes \Gamma$. Whenever $\pi : \Gamma \recht \cU(\cK)$ is a unitary representation, define the Hilbert space $\cH^\pi = \rL^2(M) \ot \cK$, with bimodule action given by
$$(a u_g) \cdot \xi \cdot (b u_h) = (au_g \ot \pi(g))\xi b u_h \quad\text{for all}\;\; a,b \in Q, g,h \in \Gamma \; .$$
If the unitary representation $\pi$ is weakly contained in the unitary representation $\rho$, then $\cH^\pi$ is weakly contained in $\cH^\rho$. If $\pi=\lambda$ is the regular representation of $\Gamma$ on $\ell^2(\Gamma)$, then $\cH^\lambda \cong \rL^2(M) \ot_Q \rL^2(M)$. When $Q$ is injective, $\bim{Q}{\rL^2(Q)}{Q}$ is weakly contained in $\bim{Q}{(\rL^2(Q) \ot \rL^2(Q))}{Q}$ and hence,
$$\rL^2(M) \ot_Q \rL^2(M) \cong \rL^2(M) \ot_Q \rL^2(Q) \ot_Q \rL^2(M)$$
is weakly contained in $\rL^2(M) \ot_Q (\rL^2(Q) \ot \rL^2(Q)) \ot_Q \rL^2(M) = \rL^2(M) \ot \rL^2(M)$. So, for $Q$ injective, $\cH^\lambda$ is weakly contained in the coarse $M$-$M$-bimodule. Finally, if $Q$ is injective and if the unitary representation $\pi$ is weakly contained in the regular representation, then $\cH^\pi$ is weakly contained in the coarse $M$-$M$-bimodule.

A next type of bimodule arises from \emph{Jones' basic construction} \cite{Jo83}. Let $M$ be a tracial von Neumann algebra with von Neumann subalgebra $P$. Denote by $\langle M,e_P \rangle$ the von Neumann algebra acting on $\rL^2(M)$ generated by $M$ and by the orthogonal projection $e_P$ of $\rL^2(M)$ onto $\rL^2(P)$. Equivalently, $\langle M,e_P \rangle$ is the commutant of the right $P$-action on $\rL^2(M)$. Given a tracial state $\tau$ on $M$, the von Neumann algebra $\langle M,e_P \rangle$ carries a natural normal semi-finite faithful trace $\Tr$ characterized by
$$\Tr (a e_P b) = \tau(ab) \quad\text{for all}\;\; a,b \in M \; .$$
In particular, we can write the Hilbert space $\rL^2(\langle M,e_P \rangle)$ and consider it as an $M$-$M$-bimodule. Then,
$$\bim{M}{\rL^2(\langle M,e_P \rangle)}{M} \cong \bim{M}{(\rL^2(M) \ot_P \rL^2(M))}{M} \; .$$
Again, if $P$ is injective, it follows that $\bim{M}{\rL^2(\langle M,e_P \rangle)}{M}$ is weakly contained in the coarse $M$-$M$-bimodule.

\subsection{Relative property (T) for an inclusion of finite von Neumann algebras}

We recall from  \cite[Proposition 4.1]{PBetti} the following definition of relative
property (T) for an inclusion of tracial von Neumann algebras.

\begin{definition} \label{def.relative-T}
Let $(M,\tau)$ be a tracial von Neumann algebra and let $P \subset M$ be a von Neumann subalgebra.
The inclusion $P \subset M$ is said to have the relative property (T) if the following
property holds: for every $\eps > 0$, there exists a finite subset
$\cJ \subset M$ and a $\delta > 0$ such that whenever $\bim{M}{\cH}{M}$ is
an $M$-$M$-bimodule admitting a unit vector $\xi$ with the properties
\begin{itemize}
\item $\| a \cdot \xi - \xi \cdot a \| < \delta$ for all $a \in \cJ$,
\item $| \langle \xi, a \cdot \xi \rangle - \tau(a)| < \delta$ and
$|\langle \xi , \xi \cdot a \rangle - \tau(a) | < \delta$ for all
$a$ in the unit ball of $M$,
\end{itemize}
there exists a vector $\xi_0 \in \cH$ satisfying $\|\xi - \xi_0\| < \eps$
and $a \cdot \xi_0 = \xi_0 \cdot a$ for all $a \in P$.
\end{definition}

Note that if $\Gamma_0$ is a subgroup of the countable group
$\Gamma$, then the inclusion $\rL(\Gamma_0) \subset \rL(\Gamma)$ has
relative property (T) if and only if $\Gamma_0 < \Gamma$ has the
relative property (T) of Kazhdan-Margulis \cite[Proposition
5.1]{PBetti}.

A normal completely positive map $\vphi : M \recht M$ is said to be subunital if $\vphi(1) \leq 1$ and subtracial if $\tau \circ \vphi \leq \vphi$. Let $P \subset M$ be an inclusion of tracial von Neumann algebras. Relative property (T) then has the following equivalent characterization: whenever $\vphi_n : M \recht M$ is a sequence of subunital and subtracial normal completely positive maps satisfying $\|x - \vphi_n(x) \|_2 \recht 0$ for all $x \in M$, we have that $\|x - \vphi_n(x)\|_2 \recht 0$ uniformly on the unit ball of $P$.

\subsection{Intertwining by bimodules}

To fix notations, we briefly recall the
\emph{intertwining-by-bimodules} technique from \cite[Section 2]{P1}
(see also \cite[Appendix C]{VBour}).

Let $(M,\tau)$ be a tracial von Neumann algebra and assume that
$A,B \subset \M_n(\C) \ot M$ are possibly non-unital von Neumann subalgebras.
Denote their respective units by $1_A$ and $1_B$. Then, the following two conditions are equivalent.
\begin{itemize}
\item $1_A (\M_n(\C) \ot \rL^2(M)) 1_B$ admits an $A$-$B$-subbimodule
that is finitely generated as a right $B$-module.
\item There is no sequence of unitaries $u_n \in \cU(A)$ satisfying
$\|E_B(x u_n y^*)\|_2 \recht 0$ for all $x,y \in 1_B(\M_n(\C) \ot M)1_A$.
\end{itemize}
If one of these equivalent conditions hold, we write $A \embed{M} B$. Otherwise, we write $A \notembed{M} B$.

When $M$ is a II$_1$ factor and $A,B \subset M$ are \emph{Cartan subalgebras,} then $A \embed{M} B$ if and only if there exists a unitary $u \in \cU(M)$ such that $A = u B u^*$, see \cite[Theorem A.1]{PBetti} (see also \cite[Theorem C.3]{VBour}).

\subsection{Cocycle superrigidity}\label{subsec.cocycle}

Let $\Gamma \actson (X,\mu)$ be a p.m.p.\ action. If $\cL$ is a Polish group, an $\cL$-valued $1$-cocycle is a measurable map $\om : \Gamma \times X \recht \cL$ satisfying $\om(g h , x) = \om(g,h \cdot x) \om(h,x)$ for all $g,h \in \Gamma$ and almost all $x \in X$. We say that $\Gamma \actson (X,\mu)$ is \emph{$\cL$-cocycle superrigid} if every $1$-cocycle with values in $\cL$ is cohomologous to a group morphism $\Gamma \recht \cL$. More precisely, this means that there exists a measurable map $\vphi : X \recht \cL$ and a group morphism $\delta : \Gamma \recht \cL$ satisfying $\om(g,x) = \vphi(g \cdot x) \delta(g) \vphi(x)^{-1}$ for all $g \in \Gamma$ and almost all $x \in X$.

We say that $\Gamma \actson (X,\mu)$ is \emph{$\Ufin$-cocycle superrigid} if it is $\cL$-cocycle superrigid for all $\cL$ in the class $\Ufin$ of Polish groups that can be realized as the closed subgroup of the unitary group of a II$_1$ factor with separable predual. Note that $\Ufin$ contains all countable groups and all compact second countable groups. Using the distance given by the $\|\,\cdot\,\|_2$-norm every group in $\Ufin$ admits a separable complete bi-invariant metric implementing the topology.

\section{Two transfer lemmas}

We want to establish W$^*$-superrigidity for certain known group
actions $\Gamma \actson (X,\mu)$. Starting from a ``mysterious''
alternative group measure space decomposition $\rL^\infty(Y) \rtimes \Lambda$ of the given factor $\rL^\infty(X) \rtimes
\Gamma$, we need
to transfer certain properties of $\Gamma \actson X$ to properties
of the unknown group $\Lambda$ or of its unknown action $\Lambda
\actson Y$. For our purposes, this transfer of properties is
achieved in Lemmas \ref{lemma.crucial} and \ref{lemma.spectralgap}.

\begin{lemma} \label{lemma.crucial}
Let $M$ be a II$_1$ factor with the following deformation property.
We are given a sequence $\vphi_n$ of subunital, subtracial, normal,
completely positive maps from $M$ to $M$ such that $\|x - \vphi_n(x)\|_2 \recht 0$
for all $x \in M$.

Let $P \subset M$ be a von Neumann subalgebra and assume that $P$ is injective.

Assume that the countable group $\Lambda$ acts trace preservingly on
the injective von Neumann algebra $Q$ and suppose that we have
identified $Q \rtimes \Lambda$ with $p M p$ for some projection $p
\in M$. We denote by $(v_s)_{s \in \Lambda}$ the canonical unitaries
in $Q \rtimes \Lambda$.

Finally, assume that $M_0 \subset M$ is a von Neumann subalgebra with
the relative property (T) and such that $M_0$ has no injective direct summand.

Then, for every $\eps > 0$, there exists $n$ and a sequence $(s_k)_k$ in $\Lambda$ such that
\begin{enumerate}
\item $\|\vphi_n(v_{s_k}) - v_{s_k}\|_2 \leq \eps$ for all $k$,
\item $\|E_P(x v_{s_k} y)\|_2 \recht 0$ for all $x,y \in M$.
\end{enumerate}
\end{lemma}

\begin{proof}
We first argue that we may actually assume that $p=1$. Put $N = p M p$
and define $\psi_n : N \recht N : \psi_n(a) = p \vphi_n(a) p$. Then,
$\psi_n$ is a sequence of subunital, subtracial, normal, completely
positive maps and $\|a - \psi_n(a) \|_2 \recht 0$ for all $a \in N$.
Let $\eps > 0$. Since $\|p - \vphi_n(p)\|_2 \recht 0$, we can take $n_0$
such that $\|\psi_n(a) - \vphi_n(a)\|_2 \leq \frac{\eps}{2} \|a\|$ for
all $a \in N$ and all $n \geq n_0$. We now replace $M$ by $N$ and $\vphi_n$
by $\psi_n$. Since $M_0$ is diffuse and $M$ is a factor, we may assume that
$p \in M_0$ and finally replace $M_0$ by $p M_0 p \subset N$.

So, for the rest of the proof, we assume that $p = 1$. We have $M = Q \rtimes \Lambda$.

Define the positive-definite functions $\vphitil_n : \Lambda \recht \C : \vphitil_n(s) := \tau(v_s^* \vphi_n(v_s))$. Note that $\vphitil_n \recht 1$ pointwise. We have to prove the following statement: for every $\eps > 0$, there exists $n$ and a sequence $(s_k)_k$ in $\Lambda$ with the properties
\begin{equation}\label{eq.prop}
|\vphitil_n(s_k) - 1 | \leq \eps \;\;\text{for all}\; k \quad\text{and}\quad \|E_P(x v_{s_k} y)\|_2 \recht 0 \;\;\text{for all}\; x,y \in M \; .
\end{equation}
Suppose that this statement is false. Fix $\eps > 0$ such that for every $n$, it is impossible to find a sequence $(s_k)_k$ in $\Lambda$ with properties \eqref{eq.prop}.

Define for every $n$, the normal completely positive map $\theta_n : M \recht M$ satisfying $\theta_n(b v_s) = \vphitil_n(s) b v_s$ for all $b \in Q$, $s \in \Lambda$. Note that $\|\theta_n(x) - x \|_2 \recht 0$ for all $x \in M$. Since $M_0 \subset M$ has the relative property (T), fix $n$ such that $\|\theta_n(w) - w \|_2^2 \leq \eps^2/2$ for all $w \in \cU(M_0)$.

By assumption, it is impossible to find a sequence $(s_k)_k$ in $\Lambda$ with properties \eqref{eq.prop}. Define
$$\cV = \{s \in \Lambda \mid |\vphitil_n(s) - 1 | \leq \eps \} \; .$$
It follows that we can take a finite subset $\cF \subset M$ and a $\delta > 0$ such that for all $s \in \cV$, we have
$$\sum_{x,y \in \cF} \|E_P(x v_s y^*)\|_2^2 \geq 2\delta \; .$$
Consider the Hilbert space $\cK = \rL^2(\langle M,e_P \rangle)$ and the unitary representation
$$\pi : \Lambda \recht \cU(\cK) : \pi(s) \xi = v_s \xi v_s^* \; .$$
Since $P$ is injective, the unitary representation $\pi$ is weakly contained in the regular representation of $\Lambda$. Define $\xi_0 \in \cK$ by the formula $\xi_0 = \sum_{x \in \cF} x^* e_P x$. Define the positive definite function $$\psi : \Lambda \recht [0,+\infty) : \psi(s) = \langle \xi_0 ,\pi(s) \xi_0 \rangle \; .$$
Note that $\psi(s) = \sum_{x,y \in \cF} \|E_P(x v_s y)\|_2^2$ and hence, $\psi(s) \geq 2\delta$ for all $s \in \cV$.

Define the normal completely positive map $\rho : M \recht M$
such that $\rho(b v_s) = \psi(s) b v_s$ for all $b \in Q$, $s \in \Lambda$.
Since $Q$ is injective and since $\pi$ is weakly contained in the regular
representation of $\Lambda$, it follows that the $M$-$M$-bimodule $\bim{M}{\cK^\rho}{M}$
defined by $\rho$ is weakly contained in the coarse $M$-$M$-bimodule
$\bim{M}{(\rL^2(M) \ot \rL^2(M))}{M}$.

We claim that $\tau(w^* \rho(w)) \in [\delta,1]$ for all $w \in \cU(M_0)$. Since $\|\theta_n(w) - w \|_2^2 \leq \eps^2/2$ for all $w \in \cU(M_0)$, it suffices to prove that $\tau(w^* \rho(w)) \in [\delta,1]$ for every unitary $w \in \cU(M)$ satisfying $\|\theta_n(w) - w \|_2^2 \leq \eps^2/2$. Take such a unitary $w$ and write $w = \sum_{s \in \Lambda} b_s v_s$. It follows that
\begin{align*}
\frac{\eps^2}{2} & \geq \|\theta_n(w) - w \|_2^2 = \sum_{s \in \Lambda} |\vphitil_n(s) - 1|^2 \, \|b_s\|_2^2 \\
& \geq \sum_{s \in \Lambda - \cV}  |\vphitil_n(s) - 1|^2 \, \|b_s\|_2^2 \\
& \geq \eps^2 \sum_{s \in \Lambda - \cV} \|b_s\|_2^2 \; .
\end{align*}
We conclude that $$\sum_{s \in \Lambda - \cV} \|b_s\|_2^2 \leq \frac{1}{2} \; .$$
Since $w$ is unitary, it follows that
$$\sum_{s \in \cV} \|b_s\|_2^2 \geq \frac{1}{2} \; .$$
It now follows that
\begin{align*}
\tau(w^* \rho(w)) & = \sum_{s \in \Lambda} \psi(s) \|b_s\|_2^2 \geq \sum_{s \in \cV} \psi(s) \|b_s\|_2^2 \\
& \geq 2 \delta \sum_{s \in \cV} \|b_s\|_2^2 \geq \delta \; .
\end{align*}
Hence, our claim is proven.

Since $\tau(w^* \rho(w)) \in [\delta,1]$ for all $w \in \cU(M_0)$, the $M$-$M$-bimodule $\bim{M}{\cK^\rho}{M}$ contains a non-zero $M_0$-central vector. On the other hand, the $M$-$M$-bimodule $\bim{M}{\cK^\rho}{M}$ is weakly contained in the coarse $M$-$M$-bimodule $\bim{M}{(\rL^2(M) \ot \rL^2(M))}{M}$. It follows that the coarse $M_0$-$M_0$-bimodule admits a sequence of almost $M_0$-central unit vectors. This is a contradiction with the assumption that $M_0$ has no injective direct summand.
\end{proof}

\begin{lemma} \label{lemma.spectralgap}
Let $M$ be a II$_1$ factor with the following deformation property. We have an inclusion $M \subset \Mtil$ of $M$ into the finite von Neumann algebra $\Mtil$ such that the $M$-$M$-bimodule $\bim{M}{\bigl(\rL^2(\Mtil) \ominus \rL^2(M)\bigr)}{M}$ is weakly contained in the coarse $M$-$M$-bimodule $\bim{M}{(\rL^2(M) \ot \rL^2(M))}{M}$. We are given a sequence $\al_n \in \Aut(\Mtil)$ such that $\|a - \al_n(a)\|_2 \recht 0$ for all $a \in \Mtil$.

Let $P \subset M$ be a von Neumann subalgebra and assume that $P$ is injective.

Assume that the countable group $\Lambda$ acts trace preservingly on the injective von Neumann algebra $Q$ and suppose that we have identified $Q \rtimes \Lambda = p M p$ for some projection $p \in M$. We denote by $(v_s)_{s \in \Lambda}$ the canonical unitaries in $Q \rtimes \Lambda$.

Finally, assume that $M_1,M_2 \subset M$ are von Neumann subalgebras such that $M_1$ and $M_2$ commute and $M_1,M_2$ have no injective direct summand.

Then, for every $\eps > 0$, there exists $n$ and a sequence $s_k \in \Lambda$ such that
\begin{enumerate}
\item $\|\al_n(v_{s_k}) - E_M(\al_n(v_{s_k}))\|_2 \leq \eps$ for all $k$,
\item $\|E_P(x v_{s_k} y) \|_2 \recht 0$ for all $x,y \in M$.
\end{enumerate}
\end{lemma}

\begin{proof}
We use throughout the trace $\tau$ normalized in such a way that $\tau(p) = 1$. We use the notation $\| \cdot  \|_2$ accordingly.

Assume that the statement is false. Take $\eps > 0$ such that for all $n \in \N$, it is impossible to find a sequence $s_k \in \Lambda$ satisfying properties 1 and 2 above.

Put $N := Q \rtimes \Lambda$, which we identified with $p M p$. Define the normal, unital $*$-homomorphism $\beta : N \recht N \ovt N$ given by $\beta(a v_s) = av_s \ot v_s$ for all $a \in Q$, $s \in \Lambda$. Put $\Ntil := p \Mtil p$. Define the $N$-$N$-bimodule $\bim{N}{\cH}{N}$ given by $\cH := \rL^2(N) \ot \rL^2(\Ntil)$ with $x \cdot \xi \cdot y := \beta(x) \xi \beta(y)$. Put $\cH^0 := \rL^2(N) \ot (\rL^2(\Ntil) \ominus \rL^2(N))$. Note that $\cH^0 \subset \cH$ is an $N$-$N$-subbimodule. By our assumptions and because $Q$ is injective, the $N$-$N$-bimodule $\bim{N}{\cH^0}{N}$ is weakly contained in the coarse $N$-$N$-bimodule.

Since $M_1$ is diffuse and $M$ is a factor, we may assume that $p \in M_1$. We replace $M_1$ by $pM_1 p$ and $M_2$ by $M_2 p$ so that they become commuting subalgebras of $N$.

Put $\eps_1 = \frac{1}{10} \eps$. Since $M_1$ has no injective direct summand and since $\bim{M_1}{\cH^0}{M_1}$ is weakly contained in the coarse $M_1$-$M_1$-bimodule, we can take a finite subset $\cF \subset M_1$ and a $\rho > 0$ such that whenever $\xi \in \cH^0$ and $\|x \cdot \xi - \xi \cdot x \|_2 \leq \rho$ for all $x \in \cF$, then $\|\xi\|_2 \leq \eps_1$.

Define the linear map $$\theta_n : M_2 \recht \cH : \theta_n(a) = (1 \ot p)(\id \ot \al_n)\beta(a)(1 \ot p) \; .$$
Take $n$ large enough such that
$$\|\al_n(p) - p \|_2 \quad\text{and}\quad \| (\id \ot \al_n^{-1})\beta(x) - \beta(x) \|_2 \;\; , x \in \cF \; \; ,$$
are all sufficiently small in order to ensure that
$$
\|x \cdot \theta_n(a) - \theta_n(a) \cdot x \|_2 \leq \rho \quad\text{for all}\;\; x \in \cF \; , \; a \in \cU(M_2) \; .
$$
It follows that
$$\|\theta_n(a) - (\id \ot E_M)\theta_n(a) \|_2 \leq \eps_1 \quad\text{for all}\;\; a \in \cU(M_2) \; .$$
When choosing $n$, we can make sure that $\|\al_n(p) - p\|_2 \leq \eps_1$, yielding
$$\|\theta_n(a) - (\id \ot \al_n)\beta(a) \|_2 \leq 2 \eps_1 \|a\| \quad\text{for all}\;\; a \in M_2 \; .$$
As a conclusion, we get
\begin{equation}\label{eq.est}
\|(\id \ot \al_n)\beta(a) - (\id \ot E_M \circ \al_n)\beta(a) \|_2 \leq 5 \eps_1 \quad\text{for all}\;\; a \in \cU(M_2) \; .
\end{equation}

Define $\cV = \{s \in \Lambda \mid \|\al_n(v_s) - E_M(\al_n(v_s))\|_2 \leq \eps \}$. By our assumption ex absurdo, it is impossible to find a sequence $s_k \in \Lambda$ satisfying conditions 1 and 2 in the formulation of the lemma. Hence, we find a finite subset $\cF_1 \subset M$ and a $\delta > 0$ such that for all $s \in \cV$, we have
$$\sum_{x,y \in \cF_1} \|E_P(x v_s y^*)\|_2^2 \geq \delta \; .$$
Define the $N$-$N$-bimodule $\bim{N}{\cK}{N}$ where
$$\cK = \rL^2(N) \ot p \rL^2( \langle M,e_P \rangle) p \quad\text{and}\quad x \cdot \xi \cdot y = \beta(x) \xi \beta(y) \; .$$ Since $P$ is injective, $\bim{N}{\bigl( p \rL^2( \langle M,e_P \rangle) p \bigr)}{N}$ is contained in the coarse $N$-$N$-bimodule. Since $Q$ is injective, also $\bim{N}{\cK}{N}$ is weakly contained in the coarse $N$-$N$-bimodule. We now prove that $\cK$ admits a non-zero $M_2$-central vector. Since $M_2$ has no injective direct summand, this yields the required contradiction.

Define the vector $\xi \in \cK$ by the formula
$$\xi := \sum_{x \in \cF_1} 1 \ot p x^* e_P x p \; .$$
We will prove that
\begin{equation}\label{eq.aim}
\langle a \cdot \xi \cdot a^*, \xi \rangle \geq \frac{3}{4} \delta  \quad\text{for all}\;\; a \in \cU(M_2) \; .
\end{equation}
Take $a \in \cU(M_2)$ and write $a = \sum_{s \in \Lambda} a_s v_s$ with $a_s \in Q$. Then,
$$(\id \ot \al_n)\beta(a) - (\id \ot E_M \circ \al_n)\beta(a) = \sum_{s \in \Lambda} a_s v_s \ot \bigl(\al_n(v_s) - E_M(\al_n(v_s))\bigr) \; .$$
Using \eqref{eq.est} and the definition of $\cV \subset \Lambda$, it follows that
\begin{align*}
(5 \eps_1)^2 & \geq \| (\id \ot \al_n)\beta(a) - (\id \ot E_M \circ \al_n)\beta(a) \|_2^2 \\
& = \sum_{s \in \Lambda} \|a_s\|_2^2 \; \| \al_n(v_s) - E_M(\al_n(v_s)) \|_2^2 \\
& \geq \eps^2 \sum_{s \in \Lambda \setminus \cV} \|a_s\|_2^2 \; .
\end{align*}
Using our definition of $\eps_1$, we conclude that
$$\sum_{s \in \Lambda \setminus \cV} \|a_s\|_2^2 \leq \frac{1}{4} \quad\text{and hence}\quad \sum_{s \in \cV} \|a_s\|_2^2 \geq \frac{3}{4} \; .$$
It follows that
\begin{align*}
\langle a \cdot \xi \cdot a^* , \xi \rangle & = \sum_{s \in \Lambda} \|a_s\|_2^2 \; \Bigl( \sum_{x,y \in \cF_1} \|E_P(x v_s y^*)\|_2^2 \Bigr) \\
& \geq \delta \sum_{s \in \cV} \|a_s\|_2^2 \geq \frac{3}{4} \delta \; .
\end{align*}
So, we have shown \eqref{eq.aim}. It follows that the unique vector of minimal norm in the closed convex hull of $\{a \cdot \xi_0 \cdot a^* \mid a \in \cU(M_2)\}$ is non-zero and $M_2$-central. This ends the proof of the lemma.
\end{proof}

\section{Some Herz-Schur multipliers on amalgamated free products} \label{sec.Schur}

Let $\Gamma$ be a countable group. A function $\vphi : \Gamma \recht \C$ is called a \emph{Herz-Schur multiplier} if $u_g \mapsto \vphi(g) u_g$ extends to an ultraweakly continuous, completely bounded linear map $\m_\vphi : \rL(\Gamma) \recht \rL(\Gamma)$. The linear space of Herz-Schur multipliers on $\Gamma$ is denoted by $\rB_2(\Gamma)$. Whenever $\vphi \in \rB_2(\Gamma)$, put $\cb{\vphi} := \cb{\m_\vphi}$.

Let $\Gamma = \Gamma_1 *_\Sigma \Gamma_2$ be an amalgamated free product. All elements $g \in \Gamma$ have a natural \emph{length} $|g|$, arising by writing $g$ as an alternating product of elements $\Gamma_1 - \Sigma$ and elements in $\Gamma_2 - \Sigma$. By convention, $|g|=0$ if and only if $g \in \Sigma$.

The following result is probably well known. It is an immediate consequence of \cite[Proposition 3.2]{Boz-Pic} and we include a proof for the convenience of the reader.

\begin{lemma} \label{lemma.multiplier}
1.\ For every $K \in \N$, there exists $K' > K$ and $\psi \in \rB_2(\Gamma)$ with the following properties.
\begin{itemize}
\item $\cb{\psi} \leq 2$.
\item $\psi(g) = 1$ if $|g| \leq K$ and $\psi(g) = 0$ if $|g| \geq K'$.
\item $0 \leq \psi(g) \leq 1$ for all $g \in \Gamma$.
\end{itemize}

2.\ For every $K \in \N$, there exists $K' > K$ and $\vphi \in \rB_2(\Gamma)$ with the following properties.
\begin{itemize}
\item $\cb{\vphi} \leq 3$.
\item $\vphi(g) = 0$ if $|g| \leq K$ and $\vphi(g) = 1$ if $|g| \geq K'$.
\item $0 \leq \vphi(g) \leq 1$ for all $g \in \Gamma$.
\end{itemize}
\end{lemma}
\begin{proof}
1.\ Whenever $0 < \rho < 1$, define $\theta_\rho(g) = \rho^{|g|}$. For all $n \in \N$, put $\gamma_n(g) = 1$ if $|g|=n$ and $\gamma_n(g) = 0$ if $|g| \neq n$. By \cite[Proposition 3.2]{Boz-Pic}, all $\theta_\rho$ and $\gamma_n$ belong to $\rB_2(\Gamma)$ and satisfy $\cb{\theta_\rho} = 1$ and $\cb{\gamma_n} \leq 4n+1$.

Choose $K \in \N$. Take $0 < \rho < 1$ close enough to $1$ such that
$$\sum_{n=0}^K (1-\rho^n)(4n+1) \leq \frac{1}{2} \; .$$
Next, take $K' > K$ large enough such that
$$\sum_{n=K'}^\infty \rho^n (4n+1) \leq \frac{1}{2} \; .$$
Define
$$\psi(g) = \begin{cases} 1 & \;\;\text{if}\;\; |g| \leq K \; , \\
\rho^k & \;\;\text{if}\;\; K < |g| < K' \; , \\
0 & \;\;\text{if}\;\; |g| \geq K' \; . \end{cases}$$
Since $$\psi = \theta_\rho + \sum_{n=0}^K (1 - \rho^n) \gamma_n - \sum_{n = K'}^\infty \rho^n \gamma_n \; ,$$ we conclude that $\psi \in \rB_2(\Gamma)$ and $\cb{\psi} \leq 2$.

2.\ Take $\psi$ as in 1 and put $\vphi(g) = 1 - \psi(g)$.
\end{proof}

\section{Factors with unique group measure space Cartan subalgebra}\label{sec.unique-gms-cartan}

\begin{definition}\label{def.G}
We define the family $\cG$ of groups $\Gamma$ of the form $\Gamma = \Gamma_1 *_\Sigma \Gamma_2$ with the following properties.
\begin{enumerate}
\item $\Gamma_1$ contains a non-amenable subgroup with the relative property (T) or $\Gamma_1$ contains two non-amenable commuting subgroups,
\item $\Sigma$ is amenable and $\Gamma_2 \neq \Sigma$,
\item There exist $g_1,\ldots,g_k \in \Gamma$ such that$\quad\dis \bigcap_{i=1}^k g_i \Sigma g_i^{-1}\quad\text{is finite}\; .$
\end{enumerate}
\end{definition}

Our main result says that all group measure space II$_1$ factors with
groups $\Gamma \in \cG$ have a unique group measure space Cartan subalgebra.
We also deal with amplifications. Therefore, denote by $\D_n(\C)$ the subalgebra
of diagonal matrices in $\M_n(\C)$. The following is a more general version
of Theorem \ref{thm.uniqueCartan-intro}.

\begin{theorem} \label{thm.uniqueCartan}
Let $\Gamma$ be a group in the family $\cG$ and $\Gamma \actson (X,\mu)$ a free ergodic p.m.p.\ action. Denote $M = \rL^\infty(X) \rtimes \Gamma$. Whenever $\Lambda \actson (Y,\eta)$ is a free ergodic p.m.p.\ action, $p \in \M_n(\C) \ot M$ is a projection and
$$\pi : \rL^\infty(Y) \rtimes \Lambda \recht p (\M_n(\C) \ot M)p$$
is an isomorphism, there exists a projection $q \in \D_n(\C) \ot \rL^\infty(X)$ and a unitary $u \in q(\M_n(\C) \ot M)p$ such that
$$\pi(\rL^\infty(Y)) = u^* (\D_n(\C) \ot \rL^\infty(X)) u \; .$$
\end{theorem}

In the rest of this section, we prove Theorems \ref{thm.uniqueCartan} and \ref{thm.no-Cartan}. We already deduce the following result, similar to \cite[Theorem 1.2]{PV3} which dealt with certain free product groups $\Gamma = \Gamma_1 * \Gamma_2$.

\begin{corollary}\label{cor.Sfactor}
Let $\Gamma = \Gamma_1 *_\Sigma \Gamma_2$ be a group in the family $\cG$. Assume that $\Gamma_1$ and $\Gamma_2$ are finitely generated  and assume that at least one of the $\Gamma_i$ has fixed prize with cost strictly larger than $1$ (e.g.\ $\Gamma_2 = \F_n$ for $2 \leq n < \infty$).

For any free ergodic p.m.p.\ action $\Gamma \actson (X,\mu)$, the II$_1$ factor $\rL^\infty(X) \rtimes \Gamma$ has trivial fundamental group. In other words, using the notation of \cite{PV3}, we have $\cS_{\text{\rm factor}}(\Gamma) = \{\{1\}\}$.
\end{corollary}
\begin{proof}
Theorem \ref{thm.uniqueCartan} implies that the fundamental group of $\rL^\infty(X) \rtimes \Gamma$ equals the fundamental group of the orbit equivalence relation $\cR = \cR(\Gamma \actson X)$. By \cite[Th\'{e}or\`{e}me IV.15]{G2}, $\cR$ has cost strictly between $1$ and $\infty$. It then follows from \cite[Proposition II.6]{G2} that $\cR$ has trivial fundamental group.
\end{proof}

\subsection{Deformation of amalgamated free product factors} \label{sec.deformation-AFP}

Let $M_1$, $M_2$ be von Neumann algebras equipped with a faithful normal tracial state $\tau$. Assume that $P$ is a common von Neumann subalgebra of $M_1$ and $M_2$ and that the traces of $M_1$, $M_2$ coincide on $P$. Denote by $M = M_1 *_P M_2$ the amalgamated free product with respect to the unique trace preserving conditional expectations (see \cite{popa-amal} and \cite{voi}) and still denote by $\tau$ the canonical tracial state on $M$. The Hilbert space $\rL^2(M)$ can be explicitly realized as follows, where we denote $\rL^2(\Mbol_i) := \rL^2(M_i) \ominus \rL^2(P)$.
$$\rL^2(M) = \rL^2(P) \oplus \bigoplus_{i_1 \neq i_2, i_2 \neq i_3, \cdots, i_{n-1} \neq i_n} \bigl( \rL^2(\Mbol_{i_1}) \ot_P \rL^2(\Mbol_{i_2}) \ot_P \cdots \ot_P \rL^2(\Mbol_{i_n}) \bigr) \; .$$
For all $0 < \rho < 1$, denote by $\m_\rho \in \B(\rL^2(M))$ the operator given by multiplication with the positive scalar $\rho^n$ on $\rL^2(\Mbol_{i_1}) \ot_P \rL^2(\Mbol_{i_2}) \ot_P \cdots \ot_P \rL^2(\Mbol_{i_n})$. Actually (and this will incidentally be a consequence from the following discussion), there is a unique normal unital completely positive map $\m_\rho : M \recht M$ whose extension to $\rL^2(M)$ is the $\m_\rho$ that we have just defined.

Following \cite[Section 2.2]{IPP}, we can define the following deformation of $M$. Define, for $i=1,2$, $\Mtil_i := M_i *_P (P \ovt \rL(\Z))$. Define $\Mtil := \Mtil_1 *_P \Mtil_2$ and observe that we have a canonical identification $\Mtil = M *_P (P \ovt \rL(\F_2))$. We now define a one-parameter group of automorphisms $\al_t$ of $\Mtil$.

Whenever $\al \in \Aut(\Mtil_1)$ and $\be \in \Aut(\Mtil_2)$ are both the identity when restricted to $P$, we have a unique $\al * \be \in \Aut(\Mtil)$ simultaneously extending $\al$ and $\be$. Denote by $u$ the canonical unitary generator of $\rL(\Z)$ viewed as a subalgebra of $\Mtil_1$ and denote by $v$ the canonical unitary generator of $\rL(\Z)$ viewed as a subalgebra of $\Mtil_2$. Let $f : \T \recht (-\pi,\pi]$ be the unique map satisfying $z = \exp(i f(z))$ for all $z \in \T$. Define the self-adjoint elements $h \in \Mtil_1$ and $k \in \Mtil_2$ as $h = f(u)$ and $k = f(v)$. Put, for all $t \in \R$, $u_t = \exp(i t h)$ and $v_t = \exp(i t k)$. Then, $\Ad(u_t) \in \Aut(\Mtil_1)$ and $\Ad(v_t) \in \Aut(\Mtil_2)$ are both the identity on $P$, so that we can define $\al_t \in \Aut(\Mtil)$ as $\al_t := \Ad(u_t) * \Ad(v_t)$.

Define $\rho_t = \tau(u_t)^2 = \tau(v_t)^2$. It happens to be that
$\rho_t = \frac{\sin^2(\pi t)}{(\pi t)^2}$ and hence, if $t$
decreases from $1$ to $0$, then $\rho_t$ increases from $0$ to $1$.
By using the definition of the trace on $M$, it is easy to see that
$$E_M(\al_t(x)) = \m_{\rho_t}(x) \; .$$
This also shows that $\m_{\rho}$ is a normal
unital completely positive map on $M$ for all $0 < \rho < 1$.

Observe that, when $P$ is injective, the $M$-$M$-bimodule $\bim{M}{\bigl(\rL^2(\Mtil) \ominus \rL^2(M)\bigr)}{M}$ is weakly contained in the coarse $M$-$M$-bimodule $\bim{M}{(\rL^2(M) \ot \rL^2(M))}{M}$ (see \cite[Proposition 3.1]{C-H} for a detailed argument). So, we are then in a situation where Lemma \ref{lemma.spectralgap} can potentially be applied.

We now present one of the main technical results from \cite{IPP}. Since the statement as we need it, is not exactly formulated in \cite{IPP}, we give a sketch of proof, indicating the different steps from \cite{IPP} that are needed to prove the result. Recall that the \emph{normalizer} $\cN_N(Q)\dpr$ of a von Neumann subalgebra $Q \subset N$ is the von Neumann algebra generated by the group of unitaries $u \in N$ satisfying $uQu^* = Q$.

\begin{theorem} \label{thm.IPP}
Let $M_1, M_2$ be tracial von Neumann algebras with a common von Neumann subalgebra $P$ on which the traces coincide. Denote by $M = M_1 *_P M_2$ the amalgamated free product w.r.t.\ the trace preserving conditional expectations. Denote, for $0 < \rho < 1$, by $\m_\rho$ the completely positive map on $M$ introduced above. Let $p \in M$ be a projection and $Q \subset p M p$ a von Neumann subalgebra.

If there exists $0 < \rho < 1$ and $\delta > 0$ such that $\tau(v^* \m_\rho(v)) \geq \delta$ for all $v \in \cU(Q)$, then $Q \embed{M} P$ or $\cN_{pMp}(Q)\dpr \embed{M} M_i$ for some $i \in \{1,2\}$. As explained above, $\cN_{pMp}(Q)\dpr$ denotes the normalizer of $Q$ inside $pMp$.
\end{theorem}

\begin{proof}
Clearly, $\tau(v^* \m_\rho(v))$ increases when $\rho$ increases. So, with the notation introduced before the theorem, we can take $0 < t_0 < 1$ such that
$$\tau(v^* \al_t(v)) = \tau(v^* E_M(\al_t(v))) = \tau(v^* \m_{\rho_t}(v)) \geq \delta \quad\text{for all}\;\; v \in \cU(Q) \; , \; 0 < t \leq t_0 \; .$$
In particular, we can take $t$ of the form $t = 2^{-n}$ such that $\tau(v^* \al_t(v)) \geq \delta$ for all $v \in \cU(Q)$. Denote by $y \in p \Mtil \al_t(p)$ the unique element of minimal $\| \cdot \|_2$ in the closed convex hull of $\{v^* \al_t(v) \mid v \in \cU(Q)\}$. It follows that $y \neq 0$ because $\tau(y) \geq \delta$ and that $x y = y \al_t(x)$ for all $x \in Q$.

Assume that $Q \notembed{M} P$. We have to prove that $\cN_{pMp}(Q)\dpr \embed{M} M_i$ for some $i \in \{1,2\}$. Repeating \cite[Proof of 3.3]{IPP} (see also \cite[Step (2) of 5.6]{Houd}), we find a non-zero $z \in p \Mtil \al_1(p)$ satisfying $x z = z \al_1(x)$ for all $x \in Q$. As in \cite[Proof of 4.3]{IPP} (actually, literally repeating \cite[Step (3) of 5.6]{Houd}), we can conclude that for $i=1$ or $i=2$, we have $Q \embed{M} M_i$. Since we assumed that $Q \notembed{M} P$, \cite[Theorem 1.1]{IPP} (see also \cite[Theorem 4.6]{Houd}) implies that the normalizer $\cN_{pMp}(Q)\dpr$ can be embedded into $M_i$ inside $M$.
\end{proof}

For later use, we also record the following non-optimal inequality. As positive operators on $\rL^2(M)$, we have $(1-\m_\rho)^2 \leq 1- \m_\rho \leq 1 - \m_\rho^2$. Hence, $\|x - \m_\rho(x)\|_2^2 \leq \|x\|_2^2 - \|\m_\rho(x)\|_2^2$. On the other hand, for all $t \in \R$ and $x \in M$, we have $$\|\al_t(x) - E_M(\al_t(x))\|_2^2 = \|x\|_2^2 - \|E_M(\al_t(x))\|_2^2 = \|x\|_2^2 - \|\m_{\rho_t}(x)\|_2^2 \; .$$
It follows that
\begin{equation}\label{eq.some-estimate}
\|x - \m_{\rho_t}(x)\|_2 \leq \|\al_t(x) - E_M(\al_t(x))\|_2 \quad\text{for all}\;\; x \in M \; , \; 0 < t < 1 \; .
\end{equation}

Finally, let $\Gamma = \Gamma_1 *_\Sigma \Gamma_2$ be an amalgamated free product of groups. Assume that $\Gamma$ acts in a trace preserving way on the tracial von Neumann algebra $(A,\tau)$. Put $M = A \rtimes \Gamma$. Putting $M_i = A \rtimes \Gamma_i$ for $i=1,2$ and $P = A \rtimes \Sigma$, we have $M = M_1 *_P M_2$ in a canonical way. Moreover, for all $0 < \rho < 1$, $a \in A$, $g \in \Gamma$, we have
$$\m_\rho(a u_g) = \rho^{|g|} a u_g \; ,$$
where $|g|$ denotes the length of $g$ in the sense explained at the beginning of Section \ref{sec.Schur}.

\subsection{A combinatorial lemma}

Fix an amalgamated free product $\Gamma = \Gamma_1 *_\Sigma \Gamma_2$ and
a trace preserving action $\Gamma \actson (A,\tau)$ of $\Gamma$ on the tracial von Neumann
algebra $(A,\tau)$. Put $M = A \rtimes \Gamma$ and $P = A \rtimes \Sigma$.

For every $K \in \N$, denote by $\cP_K$ the orthogonal projection of $\rL^2(M)$
onto the closed linear span of $\{a u_g \mid |g| \leq K, a \in A\}$.

Whenever $g_0,h_0 \in (\Gamma_1 - \Sigma) \cup (\Gamma_2 - \Sigma)$, denote by
$W_{g_0,h_0}$ the subset of $\Gamma$ consisting of those $g \in \Gamma$ with
$|g| \geq 2$ admitting a reduced expression starting with $g_0$ and ending with $h_0$ (this means that any reduced expression for $g$ starts with $g_0 \sigma$ and ends with $\sigma' h_0$ for some $\sigma,\sigma' \in \Sigma)$. Denote by $\cP_{g_0,h_0}$ the orthogonal projection of $\rL^2(M)$ onto the closed linear span of $\{a u_g \mid g \in W_{g_0,h_0}, a \in A\}$.

In general, whenever $W \subset \Gamma$, denote by $\cP_W$ the orthogonal projection of $\rL^2(M)$ onto the closed linear span of $\{a u_g \mid g \in W, a \in A\}$.

\begin{lemma}\label{lemma.combinatorial}
Let $K \in \N$ and assume that $(y_k)$ is a bounded sequence in $M$ with the following properties.
\begin{itemize}
\item $y_k = \cP_K(y_k)$ for all $k$.
\item $\|E_P(x y_k z)\|_2 \recht 0$ for all $x,z \in M$.
\end{itemize}
Let $g,h \in \Gamma$ with $|g|,|h| \geq K$ and write $g,h$ as reduced expressions. Denote by $g_0$ the first letter of $g$ and by $h_0$ the last letter of $h$. Then, we can write
$$u_g y_k u_h = a_k + b_k$$
where $a_k,b_k$ are bounded sequences in $M$ satisfying the following properties.
\begin{itemize}
\item $a_k = \cP_{g_0,h_0}(a_k)$ for all $k$.
\item $\|b_k\|_2 \recht 0$.
\end{itemize}
\end{lemma}

\begin{proof}
We fix once and for all words with letters alternatingly from $\Gamma_1 - \Sigma$ and $\Gamma_2 - \Sigma$ representing the elements $g$ and $h$.

Let $(g_1,h_1), \cdots, (g_N,h_N)$ be an enumeration of all pairs of words $(g',h')$ satisfying
\begin{itemize}
\item $|g'| + |h'| \leq K$,
\item the fixed word representing $g$ ends with the subword $g'$,
\item the fixed word representing $h$ starts with the subword $h'$.
\end{itemize}
Define $W_i = gg_i^{-1} \Sigma h_i^{-1}h$ and $W = \bigcup_{i=1}^N W_i$.
Observe that $$\cP_{W_i}(x) = u_g u_{g_i}^* E_P(u_{g_i}u_g^* x u_h^* u_{h_i}) u_{h_i}^*u_h$$
for all $x \in M$. Hence, $\cP_{W_i}$ is completely bounded as a map from $M$ to $M$.
The orthogonal projections $\cP_{W_i}$ commute and hence
$$1 - \cP_W = (1-\cP_{W_1}) \cdots (1-\cP_{W_N}) \; .$$
So, $\cP_W$ is completely bounded on $M$. We put $b_k = \cP_W(u_g y_k u_h)$ and
$a_k = u_g y_k u_h - b_k$. So, $a_k$ and $b_k$ are bounded sequences in $M$ with
$u_g y_k u_h = a_k + b_k$.

First observe that
$$\|b_k\|_2^2 \leq \sum_{i=1}^N \|\cP_{W_i}(u_g y_k u_h)\|_2^2 =
\sum_{i=1}^N \|E_P(u_{g_i} y_k u_{h_i})\|_2^2 \recht 0 \; .$$
It remains to prove that $a_k = \cP_{g_0,h_0}(a_k)$ for all $k$.
But, if $r \in \Gamma$ with $|r|\leq K$ and if $grh$ admits no reduced expression
that starts with $g_0$ and ends with $h_0$, there must exist $i \in \{1,\ldots,N\}$
such that $g_i r h_i \in \Sigma$ and hence $grh \in W$. As a consequence,
whenever $y \in M$ with $y = \cP_K(y)$, we have
$$u_g y u_h - \cP_W(u_g y u_h) \in \cP_{g_0,h_0}(\rL^2(M)) \; .$$
This concludes the proof of the lemma.
\end{proof}

\subsection{Group measure space Cartan subalgebras can be intertwined into $A \rtimes \Sigma$}

Fix $\Gamma = \Gamma_1 *_\Sigma \Gamma_2$ satisfying the conditions 1 and 2 of Definition \ref{def.G}. Let $\Gamma \actson (A,\tau)$ be a trace preserving action of $\Gamma$ on the tracial injective von Neumann algebra $(A,\tau)$. Put $M = A \rtimes \Gamma$ and $P = A \rtimes \Sigma$. Note that $P$ is injective, because $A$ is injective and $\Sigma$ is amenable.

\begin{theorem} \label{thm.embedding}
Whenever $p \in M$ is a non-zero projection and $pMp = B \rtimes \Lambda$ is a crossed product decomposition where $\Lambda \actson (B,\tau)$ is a trace preserving action on the abelian von Neumann algebra $B$, then $B \embed{M} P$.
\end{theorem}

We will prove Theorem \ref{thm.embedding} by combining the transfer of rigidity lemmas \ref{lemma.crucial}, \ref{lemma.spectralgap} with the following result saying that any abelian algebra that is normalized by \lq many\rq\ unitaries of short word length, is itself uniformly of short length.

As in Section \ref{sec.deformation-AFP}, define for every $0 < \rho < 1$, the unital completely positive map $\m_\rho$ on $M$ by $\m_\rho(a u_g) = \rho^{|g|} a u_g$ for all $a \in A$, $g \in \Gamma$.

\begin{lemma}\label{lem.lengths}
Let $p \in M$ be a projection and $B \subset pMp$ an abelian von Neumann subalgebra. Put $\eps = \tau(p)/2072$ and assume that we are given $0 < \rho < 1$ and a sequence of unitaries $v_k \in pMp$ that normalize $B$ and satisfy
\begin{itemize}
\item $\|v_k - \m_\rho(v_k)\|_2 \leq \eps/2$ for all $k$,
\item $\|E_P(x v_k y)\|_2 \recht 0$ for all $x,y \in M$.
\end{itemize}
Then, there exists a $0 < \rho_0 < 1$ and a $\delta > 0$ such that $\tau(w^* \m_{\rho_0}(w)) \geq \delta$ for all $w \in \cU(B)$.
\end{lemma}

\begin{proof}
Throughout the proof we make use of the Herz-Schur multipliers provided by Lemma \ref{lemma.multiplier}. Whenever $\vphi \in \rB_2(\Gamma)$, we can extend $\m_\vphi$ to $A \rtimes \Gamma$, without increasing the cb-norm, by putting $\m_\vphi(au_g) = \vphi(g) a u_g$ for all $a \in A$ and $g \in \Gamma$.

Assume that the lemma is false.
Take $K \in \N$ such that $\rho^K \leq 1/2$. Denote as before by $\cP_K$ the orthogonal projection of $\rL^2(M)$ onto the closed linear span of $\{a u_g \mid a \in A, |g| \leq K\}$. Note that
$$\|v- \cP_K(v)\|_2 \leq 2 \|\m_\rho(v) - v \|_2$$
for all $v \in M$.

By Lemma \ref{lemma.multiplier}.1, take $K_1 > K$ and $\psi \in \rB_2(\Gamma)$ satisfying $\cb{\psi} \leq 2$, $\psi(g) = 1$ if $|g|\leq K$, $\psi(g) = 0$ if $|g| \geq K_1$ and $0 \leq \psi(g) \leq 1$ for all $g \in \Gamma$. Define $y_k := \m_\psi(v_{k})$. Since $$\|v_{k} - \cP_K(v_{k})\|_2 \leq 2 \|\m_\rho(v_{k}) - v_{k}\|_2 \leq \eps \; ,$$ we get $\|v_{k} - y_k\|_2 \leq \eps$.
So, the sequence $(y_k)$ satisfies
\begin{itemize}
\item $\|y_k\| \leq 2$ for all $k$,
\item $y_k = \cP_{K_1}(y_k)$ for all $k$,
\item $\|v_{k} - y_k \|_2 \leq \eps$.
\end{itemize}
We claim that also $\|E_P(x y_k z)\|_2 \recht 0$ for all $x,z \in M$. Since $(y_k)$ is a bounded sequence and since $A \subset P$, we may assume that $x = u_g$, $z = u_h$ for some $g,h \in \Gamma$. Denote by $\cP_0$ the orthogonal projection of $\rL^2(M)$ onto the closed linear span of $\{a u_r \mid a \in A, r \in g^{-1} \Sigma h^{-1} \}$. Then,
\begin{align*}
\|E_P(u_g y_k u_h)\|_2 &= \|\cP_0(y_k)\|_2 = \|\cP_0(\m_\psi(v_{k}))\|_2 \\ &= \|\m_\psi(\cP_0(v_{k}))\|_2 \leq \|\cP_0(v_{k})\|_2 \\ &= \|E_P(u_g v_{k} u_h)\|_2 \recht 0 \; .
\end{align*}
This proves the claim.

By Lemma \ref{lemma.multiplier}.2, take $K_2 > 2 K_1$ and $\vphi \in \rB_2(\Gamma)$ satisfying $\cb{\vphi} \leq 3$, $\vphi(g) = 0$ if $|g| \leq 2 K_1$, $\vphi(g) = 1$ if $|g| \geq K_2$ and $0 \leq \vphi(g) \leq 1$ for all $g \in \Gamma$.

Take $0 < \rho_0 < 1$ close enough to $1$ such that $\rho_0^{K_2} \geq 1/2$. By our assumption by contradiction, take a unitary $w \in \cU(B)$ such that
$$\tau(w^* \m_{\rho_0}(w)) \leq \frac{\eps^2}{8} \; .$$
It follows that $\|\cP_{K_2}(w)\|_2 \leq \eps/2$. By the Kaplansky density theorem, take
$w_0$ in the dense $*$-subalgebra
$$M_0 = \lspan \{ a u_g \mid a \in A, g \in \Gamma \}$$
with $\|w_0\| \leq 1$ and $\|w - w_0\|_2 \leq \eps/2$. Put $x = \m_\vphi(w_0)$. Then,
$$\|\m_\vphi(w) - x \|_2 = \|\m_\vphi(w - w_0)\|_2 \leq \|w-w_0\|_2 \leq \frac{\eps}{2} \; .$$
Since $\|\cP_{K_2}(w)\|_2 \leq \eps/2$, also $\|w - \m_\vphi(w)\|_2 \leq \eps/2$. So, our element $x \in M$ has the following properties.
\begin{itemize}
\item $x \in \lspan\{a u_g \mid a \in A, |g| > 2 K_1 \}$,
\item $\|x\| \leq 3$,
\item $\|w - x \|_2 \leq \eps$.
\end{itemize}

Since $B$ is abelian, $w \in \cU(B)$ and $v_{k}$ normalizes $B$, we have
$$p = v_{k} w v_{k}^* \; w \; v_{k} w^* v_{k}^* \; w^*$$
for all $k$. We now replace $v_{k}$ by $y_k$ and $w$ by $x$. We use the estimates for $\|x\|$, $\|y_k\|$, $\|v_{k} - y_k\|_2$ and $\|w-x\|_2$, to conclude that
\begin{align*}
\|p - & y_k  \; x  \; y_k^* \;  x  \; y_k  \; x^*  \; y_k^* \;  x^* \|_2 \\ & \leq \eps (1 + 3 + 2 \cdot 3 + 3 \cdot 2 \cdot 3 + 2 \cdot 3 \cdot 2 \cdot 3 + 3 \cdot 2 \cdot 3 \cdot 2 \cdot 3 + 2 \cdot 3 \cdot 2 \cdot 3 \cdot 2 \cdot 3 + 3 \cdot 2 \cdot 3 \cdot 2 \cdot 3 \cdot 2 \cdot 3) \\ &= 1036 \eps = \frac{1}{2} \tau(p) \; .
\end{align*}
It follows that for all $k$,
$$|\tau(y_k  \; x  \; y_k^* \;  x  \; y_k  \; x^*  \; y_k^* \;  x^*)| \geq \frac{1}{2}\tau(p) \; .$$
We claim however that the left-hand side tends to $0$ when $k \recht \infty$.

Since $x \in \lspan\{a u_g \mid a \in A, |g| > 2 K_1 \}$, it suffices to prove that
\begin{equation}\label{eq.affreux}
\tau(y_k \; u_g a u_{g'}  \; y_k^*  \; u_h b u_{h'}  \; y_k  \; u_r c u_{r'}  \;  y_k^*  \; u_t d u_{t'}) \recht 0
\end{equation}
for all $a,b,c,d \in A$ and for all group elements $g,g',h,h',k,k',r,r',t,t'$ having length at least $K_1$ and being chosen such that the concatenations $gg'$, $hh'$, $rr'$ and $tt'$ are reduced.

We now study the expressions
$$u_{t'} y_k u_g \;\; , \;\;  u_{g'} y_k^* u_h   \;\; , \;\;   u_{h'} y_k u_r  \;\; \text{and} \;\; u_{r'} y_k^* u_t \; .$$
According to Lemma \ref{lemma.combinatorial}, every of these four expressions can be written as a sum $a^{(i)}_k + b^{(i)}_k$, $i=1,2,3,4$, of bounded sequences satisfying the following properties: $\|b^{(i)}_k\|_2 \recht 0$ and we have $a^{(i)}_k = \cP_{g_0,h_0}(a^{(i)}_k)$ where $g_0$ is the first letter of resp.\ $t',g',h',r'$ and $h_0$ is the last letter of resp.\ $g,h,r,t$.

Then, the left hand side of \eqref{eq.affreux} equals
$$\tau\bigl( \, (a^{(1)}_k + b^{(1)}_k) \, a \, (a^{(2)}_k + b^{(2)}_k) \, b \, (a^{(3)}_k + b^{(3)}_k) \, c \, (a^{(4)}_k + b^{(4)}_k) \, d \bigr) \; .$$
Developing all the sums, the one term that only involves $a^{(i)}_k$ equals zero because there is no simplification between consecutive factors of the product and all the other terms tend to zero because $\|b^{(i)}_k\|_2 \recht 0$ and the $(a^{(i)}_k)$ are bounded.

We have reached a contradiction and hence we have proven the lemma.
\end{proof}

It is now easy to prove Theorem \ref{thm.embedding}.

\begin{proof}[Proof of Theorem \ref{thm.embedding}]
Denote by $(v_s)_{s \in \Lambda}$ the canonical unitaries in $B \rtimes \Lambda$.
Put $\eps = \tau(p)/2072$. In the situation where $\Gamma_1$ has a non-amenable subgroup $H$ with the relative property (T), we apply Lemma \ref{lemma.crucial} to the completely positive maps $\m_\rho$, $\rho \recht 1$ and the von Neumann subalgebra $M_0 = \rL(H)$. In the situation where $\Gamma_1$ has two non-amenable commuting subgroups $H_1,H_2$, we apply Lemma \ref{lemma.spectralgap} to the deformation $\al_t$, $t \recht 0$, introduced in Section \ref{sec.deformation-AFP} and the commuting subalgebras $\rL(H_1)$, $\rL(H_2)$. Combined with \eqref{eq.some-estimate}, we always find $0 < \rho < 1$ and a sequence $s_k \in \Lambda$ such that
\begin{itemize}
\item $\|v_{s_k} - \m_\rho(v_{s_k})\|_2 \leq \eps/2$ for all $k$.
\item $\|E_P(x v_{s_k} y)\|_2 \recht 0$ for all $x,y \in M$.
\end{itemize}
By Lemma \ref{lem.lengths} we get a $0 < \rho < 1$ and a $\delta > 0$ such that $\tau(w^* \m_\rho(w)) \geq \delta$ for all $w \in \cU(B)$. By Theorem \ref{thm.IPP} we have that $B \embed{M} P$ or that $\cN_{pMp}(B)\dpr \embed{M} A \rtimes \Gamma_i$ for some $i=1,2$. Since $\cN_{pMp}(B)\dpr = pMp$ and since $\Gamma_i < \Gamma$ has infinite index, the second option is impossible, concluding the proof of the theorem.
\end{proof}

\subsection{Proof of Theorems \ref{thm.uniqueCartan} and \ref{thm.no-Cartan}}

\begin{proof}[Proof of Theorem \ref{thm.uniqueCartan}]
Put $A = \M_n(\C) \ot \rL^\infty(X)$ with $\Gamma \actson A$ acting trivially on $\M_n(\C)$. In the proof, we do not write the isomorphism $\pi$ and we put $B = \rL^\infty(Y)$. So, $A \rtimes \Gamma = \M_n(\C) \ot M$ and $p$ is a projection in $A \rtimes \Gamma$ such that $p(A \rtimes \Gamma)p = B \rtimes \Lambda$. By Theorem \ref{thm.embedding}, we get an intertwining bimodule between $B$ and $A \rtimes \Sigma$. So, we also have
$$B \embed{M} \rL^\infty(X) \rtimes \Sigma \; .$$
By condition 3 in Definition \ref{def.G} and \cite[Theorem 6.16]{PV1}, it follows that $B \embed{M} \rL^\infty(X)$. Then, \cite[Theorem A.1]{PBetti} (see also \cite[Theorem C.3]{VBour}) provides the conclusion of the theorem.
\end{proof}

\begin{proof}[Proof of Theorem \ref{thm.no-Cartan}]
Since $\Gamma$ acts ergodically on $(X,\mu)$ and $\Gamma$ is ICC, it
follows that $M$ is a factor. Assume that $B$ is a group measure
space Cartan subalgebra in $pM^n p$. Repeating the previous proof,
it follows that $B \embed{M} \rL^\infty(X)$. By \cite[Theorem
4.11]{OP1}, it follows that $\Gamma \actson (X,\mu)$ is free.
\end{proof}

\subsection{II$_1$ factors with at least two group measure space Cartan subalgebras}\label{subsec.twoCartan}

By definition, the family $\cG$ consists of amalgamated free products $\Gamma = \Gamma_1 *_\Sigma \Gamma_2$, where $\Gamma_1$ satisfies a rigidity condition (\ref{def.G}.1), where $\Sigma$ is amenable and different from $\Gamma_2$ (\ref{def.G}.2) and where $\bigcap_{i=1}^k g_i \Sigma g_i^{-1}$ is finite for some $g_1,\ldots,g_k \in \Gamma$ (\ref{def.G}.3). Without this last condition, it is possible to give examples of group actions such that $M = \rL^\infty(X) \rtimes \Gamma$ admits at least two Cartan subalgebras that are non-conjugate by an automorphism of $M$.

Indeed, Connes and Jones \cite{CJ} provide examples of free ergodic p.m.p.\ actions $\Gamma \times \Sigma \actson (X,\mu)$ such that the II$_1$ factor $\rL^\infty(X) \rtimes (\Gamma \times \Sigma)$ has at least two group measure space Cartan subalgebras that are non-conjugate by an automorphism. In their construction, $\Gamma$ can be any non-amenable group and $\Sigma$ is a specific infinite amenable group. In particular, one can consider $$(\Gamma_1 * \Gamma_2) \times \Sigma = (\Gamma_1 \times \Sigma) *_\Sigma (\Gamma_2 \times \Sigma)$$
and provide examples where conditions \ref{def.G}.1 and \ref{def.G}.2 are satisfied, but condition \ref{def.G}.3 is not.

Mimicking \cite[Section 7]{OP2}, we give other examples of II$_1$ factors with at least two group measure space decompositions. The following general non-uniqueness statement is a consequence of Example \ref{ex.nonunique}. Whenever $\Gamma = H \rtimes G$ is a semi-direct product group with $H$ being infinite abelian, then $\Gamma$ admits free ergodic p.m.p.\ actions such that the corresponding group measure space II$_1$ factor has at least two non unitarily conjugate group measure space Cartan subalgebras. Whenever $G = G_1 *_\Sigma G_2$, also $\Gamma = (H \rtimes G_1) *_{H \rtimes \Sigma} (H \rtimes G_2)$. As such, we get again examples where conditions \ref{def.G}.1 and \ref{def.G}.2 are satisfied, but condition \ref{def.G}.3 is not.

\begin{example}\label{ex.nonunique}
Let $H$ be an infinite abelian group and $G \actson^\al H$ an action by automorphisms. Let $H \hookrightarrow K$ be a dense embedding of $H$ into the compact abelian group $K$. Assume that $G \actson H$ extends to an action by homeomorphisms of $K$ that we still denote by $\al$. Whenever $G \actson (X,\mu)$ is a free ergodic p.m.p.\ action, consider the free ergodic p.m.p.\ action
$$H \rtimes G \actson K \times X \quad\text{given by}\quad \begin{cases} h \cdot (k,x) = (h + k,x) \\ g \cdot (k,x) = (\al_g(k),g \cdot x)\end{cases} \quad\text{for all}\;\; h \in H, g \in G, k \in K, x \in X \; .$$
Dualizing the embedding $H \hookrightarrow K$, we get the embedding $\widehat{K} \hookrightarrow \widehat{H}$ and the action of $G$ by automorphisms of $\widehat{K}$ and $\widehat{H}$. We canonically have
$$\rL^\infty(K \times X) \rtimes (H \rtimes G) = \rL^\infty(\widehat{H} \times X) \rtimes (\widehat{K} \rtimes G) \; .$$
First of all, the group measure space Cartan subalgebras $\rL^\infty(K \times X)$ and $\rL^\infty(\widehat{H} \times X)$ are never \emph{unitarily conjugate.} Indeed, if $h_n \in H$ is a sequence tending to infinity in $H$, the unitaries $u_{h_n} \in \rL(H)$ satisfy $\|E_{\rL^\infty(K \times X)}(a u_{h_n} b)\|_2 \recht 0$ for all $a,b \in \rL^\infty(K \times X) \rtimes (H \rtimes G)$. It follows that $\rL^\infty(\widehat{H}) = \rL(H) \not\prec \rL^\infty(K \times X)$. A fortiori, $\rL^\infty(\widehat{H} \times X)$ cannot be unitarily conjugated onto $\rL^\infty(K \times X)$.

In certain examples, the Cartan subalgebras $\rL^\infty(K \times X)$ and $\rL^\infty(\widehat{H} \times X)$ are \emph{conjugate by an automorphism.} This is, for instance, always the case when the action $G \actson H$ is trivial. Then, the crossed product II$_1$ factor is the tensor product of $\rL^\infty(X) \rtimes G$ and $\rL^\infty(K) \times H = \rL^\infty(\widehat{H}) \rtimes \widehat{K}$. The second tensor factor is the hyperfinite II$_1$ factor and hence, the Cartan subalgebras $\rL^\infty(K)$ and $\rL^\infty(\widehat{H})$ are conjugate by an automorphism \cite{OW,CFW}.

In other examples, the Cartan subalgebras $\rL^\infty(K \times X)$ and $\rL^\infty(\widehat{H} \times X)$ are \emph{non conjugate by an automorphism.} Consider $G = \SL(n,\Z) \actson H = \Z^n$ and $\Z^n \hookrightarrow K = \Z_p^n$, where $\Z_p$ denotes the ring of $p$-adic integers for some prime number $p$. It follows that $\widehat{K} \rtimes G$ is the direct limit of a sequence of groups that are virtually isomorphic with $\SL(n,\Z)$.
\begin{itemize}
\item For $n = 2$, $H \rtimes G$ does not have the Haagerup property (because $H$ is an infinite subgroup with the relative property (T)), while $\widehat{K} \rtimes G$ has the Haagerup property (as the direct limit of groups with the Haagerup property).
\item For $n = 3$, $H \rtimes G$ has property (T), while $\widehat{K} \rtimes G$ does not have property (T) (as the direct limit of a strictly increasing sequence of groups).
\end{itemize}
Since both property (T) \cite[Corollary 1.4]{Fu1} and the Haagerup property
\cite[Remark 3.5.6$^\circ$]{PBetti} are measure equivalence invariants,
it follows that for $n = 2,3$, the group actions $(H \rtimes G) \actson (K \times X)$
and $(\widehat{K} \rtimes G) \actson (\widehat{H} \times X)$ are not stably orbit equivalent.
Hence, the corresponding group measure space Cartan subalgebras are not conjugate by an automorphism either.
\end{example}

\subsection{Amalgamated free products over groups with the Haagerup property}  \label{subsec.Haagerup}

We mention that the result in Theorem \ref{thm.uniqueCartan} also holds for certain amalgamated free products $\Gamma = \Gamma_1 *_\Sigma \Gamma_2$ over \emph{non-amenable groups $\Sigma$ with the Haagerup property,} once the free ergodic p.m.p.\ action $\Gamma \actson (X,\mu)$ is such that $\rL^\infty(X) \rtimes \Sigma$ still has the Haagerup property. This latter condition is not automatic, but holds for plain Bernoulli actions by \cite[Theorem 1.1]{CSV}. More precisely, apart from the Haagerup property of $\rL^\infty(X) \rtimes \Sigma$, one has to assume that $\Gamma = \Gamma_1 *_\Sigma \Gamma_2$ is such that $\Gamma_1$ admits an infinite subgroup with property (T), that $\Sigma$ has the Haagerup property, that $\Sigma \neq \Gamma_2$ and that the \lq malnormality\rq\ condition \ref{def.G}.3 holds.

In order to prove such a statement, it suffices to observe that Lemma \ref{lemma.crucial} still holds when $M_0$ has property (T) and $P$ has the Haagerup property. In the proof of Theorem \ref{thm.uniqueCartan}, Lemma \ref{lemma.crucial} is applied to $P = \rL^\infty(X) \rtimes \Sigma$.

\section{Stable W$^*$-superrigidity theorems} \label{sec.superrigidity}

\subsection{W$^*$-superrigidity}

Let $\Gamma \actson (X,\mu)$ be a p.m.p.\ action. Denote $A = \rL^\infty(X)$
and denote by $(\si_g)_{g \in \Gamma}$ the corresponding group of automorphisms of $A$.
We denote by $\rZ^1(\Gamma \actson X)$ the abelian group of scalar $1$-cocycles for the
action $\Gamma \actson X$, i.e.\ the group of functions $\om : \Gamma \recht \cU(A) :
g \mapsto \om_g$ satisfying $\om_{gh} = \om_g \si_g(\om_h)$ for all $g,h \in \Gamma$.

Let $\Delta : X \recht Y$ be a conjugacy between the free ergodic p.m.p.\ actions
$\Gamma \actson X$ and $\Lambda \actson Y$, with corresponding group isomorphism
$\delta : \Gamma \recht \Lambda$. Define the isomorphism
$\Delta_* : \rL^\infty(X) \recht \rL^\infty(Y) : \Delta(a) = a \circ \Delta^{-1}$.
Whenever $\om \in \rZ^1(\Gamma \actson A)$, we get an isomorphism
\begin{equation}\label{eq.form}
\theta : \rL^\infty(X) \rtimes \Gamma \recht \rL^\infty(Y) \rtimes \Lambda : \theta(a u_g) =
\Delta_*(a \om_g) \, u_{\delta(g)} \quad\text{for all}\;\; a \in \rL^\infty(X) \; , \;
g \in \Gamma \; .
\end{equation}

\begin{definition}\label{def.W-superrigid}
We call a free ergodic p.m.p.\ action $\Gamma \actson (X,\mu)$ \emph{W$^*$-superrigid}
if the following property holds. If $\Lambda \actson (Y,\eta)$ is a free ergodic p.m.p.\
action and $\pi : \rL^\infty(X) \rtimes \Gamma \recht \rL^\infty(Y) \rtimes \Lambda$ is an
isomorphism, then the groups $\Gamma$ and $\Lambda$ are isomorphic, their actions
$\Gamma \actson X$, $\Lambda \actson Y$ are conjugate and, up to a unitary conjugacy,
$\pi$ is of the form \eqref{eq.form}.
\end{definition}

Actually, the notion of \emph{stable W$^*$-superrigidity} is more
natural (see Definition \ref{def.stable-W-superrigid} below),
allowing in the correct way for amplifications. Indeed, in order for
$\Gamma \actson (X,\mu)$ to be W$^*$-superrigid in the above sense,
$\Gamma$ should not have finite normal subgroups, which is a
somewhat restrictive assumption.

We start with the following more precise version of Theorem \ref{thm.kida-intro}.

\begin{theorem}\label{thm.kida}
Let $n \geq 3$ and denote by $T_n$ the subgroup of upper triangular matrices in
$\PSL(n,\Z)$. Put $\Gamma = \PSL(n,\Z) *_{T_n} \PSL(n,\Z)$. Then, every free
ergodic p.m.p.\ action $\Gamma \actson (X,\mu)$ with the property that all finite
index subgroups of $T_n$ act ergodically on $(X,\mu)$, is W$^*$-superrigid.
In particular, all free p.m.p.\ mixing actions of $\Gamma$ are W$^*$-superrigid.
\end{theorem}

\begin{proof}
Take a free p.m.p.\ action $\Gamma \actson (X,\mu)$ such that the restriction to
every finite index subgroup of $T_n$ is ergodic. Assume that $\rL^\infty(X) \rtimes \Gamma
= \rL^\infty(Y) \rtimes \Lambda$ for some free ergodic p.m.p.\ action $\Lambda \actson (Y,\eta)$.
By Theorem \ref{thm.uniqueCartan}, we can unitarily conjugate $\rL^\infty(Y)$ onto $\rL^\infty(X)$.
Hence, $\Gamma \actson (X,\mu)$ and $\Lambda \actson (Y,\eta)$ are orbit equivalent. But then,
\cite[Theorem 1.4]{kida-amal} yields the conclusion of the theorem.
\end{proof}

\subsection{Stable W$^*$-superrigidity}

In order to introduce the more natural notion of \emph{stable W$^*$-superrigidity,}
we first discuss stable isomorphism of II$_1$ factors. If $M$ is a II$_1$ factor
and $t > 0$, the amplification $M^t$ is defined as $M^t := p(\M_n(\C) \ot M)p$, where $p \in \M_n(\C) \ot M$ is a projection satisfying $(\Tr \ot \tau)(p) = t$. Note that $M^t$ is uniquely defined up to unitary conjugacy.

\begin{definition}\label{def.W-equivalence}
Let $M$ and $N$ be II$_1$ factors. A \emph{stable isomorphism} between $M$ and $N$ is an isomorphism $\pi : N \recht M^t$ for some $t > 0$.

Given the isomorphism $\pi : N \recht M^t$ and a projection $p \in \M_n(\C) \ot M$ with $(\Tr \ot \tau)(p) = t$, define the Hilbert space $\cH^\pi := \bigl(\M_{1,n}(\C) \ot \rL^2(M)\bigr)p$. The formula $a \cdot \xi \cdot b := a \xi \pi(b)$ turns $\cH^\pi$ into an $M$-$N$-bimodule such that the right $N$-action equals the commutant of the left $M$-action and vice versa.

We equivalently define a \emph{stable isomorphism} between $M$ and $N$ as being an $M$-$N$-bimodule $\bim{M}{\cH}{N}$ such that the right $N$-action equals the commutant of the left $M$-action (and, equivalently, vice versa). Every stable isomorphism $\bim{M}{\cH}{N}$ is unitarily equivalent with $\bim{M}{\cH^\pi}{N}$ for an isomorphism $\pi : N \recht M^t$ that is uniquely determined up to unitary conjugacy.

We call the number $t > 0$ the \emph{compression constant} of the stable isomorphism $\bim{M}{\cH^\pi}{N}$.
\end{definition}

Stable isomorphisms can be composed using the Connes tensor product. Of course, we have $$\cH^\pi \ot_N \cH^\eta \cong \cH^{(\id \ot \pi)\eta} \; .$$

Fix free ergodic p.m.p.\ actions $\Gamma \actson (X,\mu)$ and $\Lambda \actson (Y,\eta)$. Stabilizing the notions of W$^*$-equivalence, orbit equivalence and conjugacy (see the beginning of Section \ref{sec.prelim}), we introduce the following terminology.

First recall that a free p.m.p.\ action $\Gamma \actson (X,\mu)$ is said to be \emph{induced} from $\Gamma_0 \actson X_0$ if $\Gamma_0 < \Gamma$ is a finite index subgroup, $X_0 \subset X$ is a non-negligible $\Gamma_0$-invariant subset and, up to measure zero, the sets $g \cdot X_0$, $g \in \Gamma/\Gamma_0$, form a partition of $X$.

\begin{itemize}
\item We call \emph{stable W$^*$-equivalence} between the actions $\Gamma \actson (X,\mu)$ and $\Lambda \actson (Y,\eta)$, any stable isomorphism between $\rL^\infty(X) \rtimes \Gamma$ and $\rL^\infty(Y) \rtimes \Lambda$.

\item We call \emph{stable orbit equivalence} between the actions $\Gamma \actson (X,\mu)$ and $\Lambda \actson (Y,\eta)$, any stable isomorphism between the orbit equivalence relations $\cR(\Gamma \actson X)$ and $\cR(\Lambda \actson Y)$, i.e.\ any isomorphism $\Delta : X_1 \recht Y_1$ between non-negligible subsets $X_1 \subset X$, $Y_1 \subset Y$ satisfying
    $$\Delta(X_1 \cap \Gamma \cdot x) = Y_1 \cap \Lambda \cdot \Delta(x)$$
    for a.e.\ $x \in X_1$. We call $\mu(X_1)/\eta(Y_1)$ the compression constant of the stable orbit equivalence.

\item We call \emph{stable conjugacy} between the actions $\Gamma \actson (X,\mu)$ and $\Lambda \actson (Y,\eta)$, any conjugacy between the actions $\frac{\Gamma_0}{G} \actson \frac{X_0}{G}$ and $\frac{\Lambda_0}{H} \actson \frac{Y_0}{H}$ where $\Gamma \actson X$, $\Lambda \actson Y$ are induced from $\Gamma_0 \actson X_0$, $\Lambda_0 \actson Y_0$ and where $G \lhd \Gamma_0$, $H \lhd \Lambda_0$ are finite normal subgroups. The number $\frac{|H| \, [\Lambda : \Lambda_0]}{|G| \, [\Gamma : \Gamma_0]}$ is called the compression constant of the stable conjugacy.
\end{itemize}

{\it From stable conjugacy to stable orbit equivalence.} If $\Gamma \actson (X,\mu)$ is induced from $\Gamma_0 \actson X_0$ and if $G \lhd \Gamma_0$ is a finite normal subgroup, the canonical stable orbit equivalence between $\Gamma \actson (X,\mu)$ and $\frac{\Gamma_0}{G} \actson \frac{X_0}{G}$ is defined by taking a fundamental domain $X_1 \subset X_0$ for the action $G \actson X_0$ and restricting the quotient map $X_0 \recht \frac{X_0}{G}$ to $X_1$. The compression constant is $(|G| \, [\Gamma : \Gamma_0])^{-1}$. The stable orbit equivalence associated with a stable conjugacy between $\Gamma \actson (X,\mu)$ and $\Lambda \actson (Y,\eta)$ is defined as the composition of the canonical stable orbit equivalences with $\frac{\Gamma_0}{G} \actson \frac{X_0}{G}$, resp.\ $\frac{\Lambda_0}{H} \actson \frac{Y_0}{H}$, together with the conjugacy between both actions.

{\it From stable orbit equivalence to stable W$^*$-equivalence.} Let $\Delta : X_1 \recht Y_1$ be a stable orbit equivalence and denote by $p \in \rL^\infty(X)$, $q \in \rL^\infty(Y)$ the projections with support $X_1,Y_1$. Then $\Delta$ gives rise to a canonical isomorphism $\pi_\Delta :  q(\rL^\infty(Y) \rtimes \Lambda)q \recht p(\rL^\infty(X) \rtimes \Gamma)p$ satisfying $\pi_\Delta(b) = b \circ \Delta$ for all $b \in \rL^\infty(Y_1)$. The isomorphism $\pi_\Delta$ amplifies to a stable W$^*$-equivalence $\rL^\infty(Y) \rtimes \Lambda \recht (\rL^\infty(X) \rtimes \Gamma)^t$, with $t = \mu(X_1)/\eta(Y_1)$, that we still denote by $\pi_\Delta$.

\begin{definition} \label{def.stable-W-superrigid}
A free ergodic p.m.p.\ action $\Gamma \actson (X,\mu)$ is said to be \emph{stably W$^*$-superrigid} if the following holds. Whenever $\pi$ is a stable W$^*$-equivalence between $\Gamma \actson (X,\mu)$ and an arbitrary free ergodic p.m.p.\ action $\Lambda \actson (Y,\eta)$, it follows that the actions are stably conjugate and that $\pi$ equals the composition of
\begin{itemize}
\item the canonical stable W$^*$-equivalence given by the stable conjugacy,
\item the automorphism of $\rL^\infty(X) \rtimes \Gamma$ given by an element of $\rZ^1(\Gamma \actson X)$,
\item an inner automorphism of $\rL^\infty(X) \rtimes \Gamma$.
\end{itemize}
\end{definition}

Let $\Gamma \actson (X,\mu)$ be stably W$^*$-superrigid. If moreover $\Gamma$ has no finite normal subgroups and if finite index subgroups of $\Gamma$ still act ergodically on $(X,\mu)$, then $\Gamma \actson (X,\mu)$ is W$^*$-superrigid in the sense of Section \ref{sec.intro}.

For certain families of group actions $\Gamma \actson (X,\mu)$, all $1$-cocycles with values in $S^1$ are known to be cohomologous to a group morphism and then we may assume that the corresponding automorphism of $\rL^\infty(X) \rtimes \Gamma$ is implemented by a character $\Gamma \recht S^1$.

{\it From stable W$^*$-equivalence to stable orbit equivalence.}
Let $\Gamma \actson (X,\mu)$ be a free ergodic p.m.p.\ action satisfying the conclusion of Theorem \ref{thm.uniqueCartan}. Whenever $\pi$ is a stable W$^*$-equivalence between $\Gamma \actson (X,\mu)$ and an arbitrary free ergodic p.m.p.\ action $\Lambda \actson (Y,\eta)$, we then find a stable orbit equivalence $\Delta$ between $\Gamma \actson X$ and $\Lambda \actson Y$ such that $\pi$ equals the composition of $\pi_\Delta$, the automorphism of $\rL^\infty(X) \rtimes \Gamma$ given by an element of $\rZ^1(\Gamma \actson X)$ and an inner automorphism.

{\it From stable orbit equivalence to stable conjugacy.}
Let $\Gamma \actson (X,\mu)$ be a free ergodic p.m.p.\ action that is cocycle superrigid with arbitrary countable target groups (see paragraph \ref{subsec.cocycle} for terminology). Whenever $\Delta$ is a stable orbit equivalence between $\Gamma \actson (X,\mu)$ and an arbitrary free ergodic p.m.p.\ action $\Lambda \actson (Y,\eta)$, it follows (see \cite[Proposition 5.11]{P0}, \cite[Lemma 4.7]{VBour}) that the actions are stably conjugate and that $\Delta$ equals the composition of the canonical stable orbit equivalence given by the stable conjugacy and an inner automorphism, i.e.\ an automorphism $\Delta_0$ of $(X,\mu)$ satisfying $\Delta_0(x) \in \Gamma \cdot x$ for a.e.\ $x \in X$. Even more so, it actually follows that there exists a finite normal subgroup $G \lhd \Gamma$ and that $\Lambda \actson Y$ is induced from $\Lambda_0 \actson Y_0$ such that the actions $\frac{\Gamma}{G} \actson \frac{X}{G}$ and $\Lambda_0 \actson Y_0$ are conjugate.

Summarizing the previous two paragraphs, we have proven the following lemma.

\begin{lemma}\label{lemma.assemble}
Let $\Gamma \actson (X,\mu)$ be a free ergodic p.m.p.\ action that satisfies the conclusion of Theorem \ref{thm.uniqueCartan} and that is cocycle superrigid with arbitrary countable target groups, as well as with target group $S^1$. Then, $\Gamma \actson (X,\mu)$ is stably W$^*$-superrigid. Even more precisely, we have the following.

Let $\Lambda \actson (Y,\eta)$ be an arbitrary free ergodic p.m.p.\ action and $\pi : \rL^\infty(X) \rtimes \Gamma \recht (\rL^\infty(Y) \rtimes \Lambda)^t$ a $*$-isomorphism for some $t > 0$. Then, $\Lambda \actson Y$ is induced from $\Lambda_0 \actson Y_0$ and there exist
\begin{itemize}
\item a finite normal subgroup $G \lhd \Gamma$ and a group isomorphism $\delta : \frac{\Gamma}{G} \recht \Lambda_0$,
\item a measure space isomorphism $\Delta : \frac{X}{G} \recht Y_0$ conjugating the actions, i.e.\ $\Delta(g \cdot x) = \delta(g) \cdot \Delta(x)$ for all $g \in \frac{\Gamma}{G}$ and a.e.\ $x \in X$,
\item a character $\om : \Gamma \recht S^1$,
\end{itemize}
such that $t = \frac{|G|}{[\Lambda: \Lambda_0]}$ and such that, after a unitary conjugacy, $\pi$ equals the composition of
\begin{itemize}
\item the automorphism $\pi_\om$ of $\rL^\infty(X) \rtimes \Gamma$ given by $\pi_\om(a u_g) = \om(g) a u_g$,
\item the canonical isomorphism $\rL^\infty(X) \rtimes \Gamma \recht \bigl(\rL^\infty\bigl(\frac{X}{G}\bigr) \rtimes \frac{\Gamma}{G}\bigr)^n$ where $n = |G|$,
\item the isomorphism $\pi_\Delta : \rL^\infty\bigl(\frac{X}{G}\bigr) \rtimes \frac{\Gamma}{G} \recht \rL^\infty(Y_0) \rtimes \Lambda_0$ given by $\pi_\Delta(au_g) = (a \circ \Delta^{-1}) u_{\delta(g)}$,
\item the canonical isomorphism $\rL^\infty(Y_0) \rtimes \Lambda_0 \recht (\rL^\infty(Y) \rtimes \Lambda)^{\frac{1}{m}}$ where $m = [\Lambda:\Lambda_0]$.
\end{itemize}
\end{lemma}

\begin{remark}\label{rem.fundamentalgroup}
Let $\Gamma \actson (X,\mu)$ be a stably W$^*$-superrigid, free ergodic p.m.p.\ action. If $\Gamma$ has no finite normal subgroups and if the finite index subgroups of $\Gamma$ act ergodically on $(X,\mu)$, then $\rL^\infty(X) \rtimes \Gamma$ has trivial fundamental group.
\end{remark}

Recall that given any action $\Gamma \actson I$ of a countable group $\Gamma$ on a countable set $I$, the \emph{generalized Bernoulli action} with base probability space $(X_0,\mu_0)$, is defined as $\Gamma \actson (X_0,\mu_0)^I$ where $(g \cdot x)_i = x_{g^{-1} \cdot i}$ for all $g \in \Gamma$, $x \in X_0^I$ and $i \in I$.

Recall that given any orthogonal representation $\pi : \Gamma \recht \rO(\cK_\R)$ of $\Gamma$ on the real Hilbert space $\cK_\R$, the Gaussian functor allows to define the \emph{Gaussian} p.m.p.\ action of $\Gamma$ on the Gaussian probability space defined by $\cK_\R$.

\begin{theorem} \label{thm.stable-superrigid}
Let $\Gamma_1,\Gamma_2$ be countable groups with a common infinite amenable subgroup $\Sigma$. Assume that $\Sigma$ is a proper, normal subgroup of $\Gamma_2$ and that there exist $g_1,\ldots,g_k \in \Gamma_1$ such that
$$
\bigcap_{i=1}^k g_i \Sigma g_i^{-1} \quad\text{is finite} \; .
$$
Put $\Gamma = \Gamma_1 *_\Sigma \Gamma_2$.
\begin{enumerate}
\item If $\Gamma_1$ admits a non-amenable normal subgroup $H$ with the relative property (T), then all of the following actions are stably W$^*$-superrigid.
\begin{itemize}
\item Every free p.m.p.\ action $\Gamma \actson (X,\mu)$ whose restriction to $\Gamma_1$ is a
generalized Bernoulli action $\Gamma_1 \actson (X_0,\mu_0)^I$ with the property that both $H \cdot i$ and $\Sigma \cdot i$ are infinite for all $i \in I$.
\item Every free p.m.p.\ action $\Gamma \actson (X,\mu)$ whose restriction to $\Gamma_1$ is a Gaussian action defined by an orthogonal representation $\pi : \Gamma_1 \recht \rO(\cK_\R)$ with the property that both restrictions $\pi_{|H}$ and $\pi_{|\Sigma}$ have no non-zero finite dimensional subrepresentations.
\end{itemize}
\item Suppose that $\Gamma_1$ admits non-amenable commuting subgroups $H$ and $H'$ such that $H$ is normal in $\Gamma_1$. If $\Gamma \actson (X,\mu)$ is a free p.m.p.\ action whose restriction to $\Gamma_1$ is a generalized Bernoulli action
$\Gamma_1 \actson (X_0,\mu_0)^I$ with the properties that $\Stab i \cap H'$ is amenable for all $i \in I$ and that both $H \cdot i$ and $\Sigma \cdot i$ are infinite for all $i \in I$, then $\Gamma \actson (X,\mu)$ is stably W$^*$-superrigid.
\end{enumerate}
In particular, for all groups $\Gamma$ mentioned in this theorem, the plain Bernoulli action $\Gamma \actson (X_0,\mu_0)^\Gamma$ is stably W$^*$-superrigid.

More precisely, all actions $\Gamma \actson X$ appearing in the theorem satisfy the conclusions of Lemma \ref{lemma.assemble}.
\end{theorem}

Before proving Theorem \ref{thm.stable-superrigid}, we provide the following concrete examples of stably W$^*$-superrigid group actions.

\begin{example} \label{ex.firstex}
\begin{enumerate}
\item Denote by $T_n < \PSL(n,\Z)$ the subgroup of upper triangular matrices and assume $n \geq 3$. For an arbitrary non-trivial group $\Lambda$ and an arbitrary infinite subgroup $\Sigma < T_n$, put $\Gamma = \PSL(n,\Z) *_{\Sigma} (\Sigma \times \Lambda)$. Whenever $\Gamma \actson I$ is such that $\Sigma \cdot i$ is infinite for all $i \in I$, the generalized Bernoulli actions $\Gamma \actson (X_0,\mu_0)^I$ are stably W$^*$-superrigid.
\item Let $\Sigma$ be an infinite amenable subgroup of the non-amenable group $H$. Assume that $H$ is finitely generated and that $\Sigma \cap \cZ(H) = \{e\}$. Let $\Lambda$ be an arbitrary non-trivial group.
Define $\Gamma_1 = H \times H$, view $\Sigma$ as a subgroup of $\Gamma_1$ diagonally and put $\Gamma = \Gamma_1 *_\Sigma \Gamma_2$. Whenever $\Gamma \actson I$ is such that $\Sigma \cdot i$ and $(H \times \{e\})\cdot i$ is infinite for all $i \in I$, the generalized Bernoulli actions $\Gamma \actson (X_0,\mu_0)^I$ are stably W$^*$-superrigid.
\end{enumerate}
\end{example}

Actually, exploiting the full strength of Theorem \ref{thm.stable-superrigid}, many more examples of stably W$^*$-superrigid actions can be given. Given an arbitrary faithful\footnote{This means that non-trivial group elements act non-trivially.} p.m.p.\ action $\Lambda \actson (X_0,\mu_0)$, we construct a stably W$^*$-superrigid, free ergodic p.m.p.\ action $\Gamma \actson (X,\mu)$ such that $\Lambda$ is a subgroup of $\Gamma$ and such that $\Lambda \actson (X_0,\mu_0)$ is a quotient of the restriction of $\Gamma \actson (X,\mu)$ to $\Lambda$.

Recall first the construction of a \emph{co-induced action.} Let $\Lambda < \Gamma$ be a subgroup and $\Lambda \actson (X_0,\mu_0)$ a p.m.p.\ action. Choose a map $\pi : \Gamma \recht \Lambda$ satisfying $\pi(gh) = \pi(g) h$ for all $g \in \Gamma, h \in \Lambda$. Define the $1$-cocycle $\om$ for the action $\Gamma \actson \Gamma/\Lambda$ with values in $\Lambda$ by the formula
$$\om(g,h\Lambda) = \pi(gh) \pi(h)^{-1} \; .$$
Different choices of $\pi$ lead to cohomologous $1$-cocycles $\om$. Define the probability space $(X,\mu) := (X_0,\mu_0)^{\Gamma/\Lambda}$ and the p.m.p.\ action $\Gamma \actson (X,\mu)$ given by
$$(g^{-1} \cdot x)_{h \Lambda} := \om(g,h\Lambda)^{-1} \cdot x_{gh\Lambda} \quad\text{for all}\;\; g,h \in \Lambda, x \in X \; .$$
We call $\Gamma \actson (X,\mu)$ the \emph{co-induced action of $\Lambda \actson (X_0,\mu_0)$ to $\Gamma$.} Different choices of $\pi$ lead to conjugate actions.

Note that we can choose $\pi$ such that $\pi(h) = h$ for all $h \in \Lambda$. Then, the quotient map $\theta : X \recht X_0 : x \mapsto x_{e \Lambda}$ satisfies $\theta(h \cdot x) = h \cdot \theta(x)$ for all $h \in \Lambda, x \in X$. Hence, $\Lambda \actson X_0$ arises as a quotient of the restriction of $\Gamma \actson X$ to $\Lambda$.

\begin{example} \label{ex.secondex}
Take $\Gamma = \PSL(n,\Z) *_{\Sigma} (\Sigma \times \Lambda)$ as in Example \ref{ex.firstex}.1 or take $\Gamma = (H \times H) *_\Sigma (\Sigma \times \Lambda)$ as in Example \ref{ex.firstex}.2. Let $\Lambda \actson (X_0,\mu_0)$ be an arbitrary faithful p.m.p.\ action. Then, the co-induced action $\Gamma \actson (X,\mu)$ is stably W$^*$-superrigid.

In the first example, put $\Gamma_1 = \PSL(n,\Z)$ and in the second example, put $\Gamma_1 = H \times H$. In order to apply Theorem \ref{thm.stable-superrigid}, it suffices to observe that the restriction of the co-induced action $\Gamma \actson (X,\mu)$ to $\Gamma_1$ is a $\Gamma_1$-Bernoulli action. Indeed, the action $\Gamma_1 \times \Lambda \actson \Gamma$ by left-right multiplication is free. Therefore, we can choose $\pi : \Gamma \recht \Lambda$ satisfying $\pi(ghk) = \pi(h)k$ for all $g \in \Gamma_1, h \in \Gamma, k \in \Lambda$. Associated with $\pi$ is the $1$-cocycle $\om$ for $\Gamma \actson \Gamma/\Lambda$, which now satisfies $\om(g,h\Lambda) = e$ for all $g \in \Gamma_1, h \in \Gamma$. Hence, the restriction $\Gamma_1 \actson (X,\mu)$ is precisely the Bernoulli action $\Gamma_1 \actson (X_0,\mu_0)^{\Gamma/\Lambda}$. The latter can be seen as the plain Bernoulli action $\Gamma_1 \actson Y_0^{\Gamma_1}$, where $Y_0 := X_0^{\Gamma_1 \backslash \Gamma / \Lambda}$.
\end{example}

\begin{example}\label{ex.Haagerup}
Let $\Gamma_1$ be an infinite group with property (T), $\Sigma < \Gamma_1$ an infinite subgroup with the Haagerup property and $\Lambda$ an arbitrary non-trivial group. Put $\Gamma = \Gamma_1 *_\Sigma (\Sigma \times \Lambda)$. Using paragraph \ref{subsec.Haagerup} instead of Theorem \ref{thm.uniqueCartan}, it follows that the plain Bernoulli action $\Gamma \actson (X_0,\mu_0)^\Gamma$ is stably W$^*$-superrigid. Similarly (cf.\ Example \ref{ex.secondex}), the co-induced action $\Gamma \actson (X,\mu)$ of an arbitrary faithful p.m.p.\ action $\Lambda \actson (X_0,\mu_0)$, follows stably W$^*$-superrigid.
\end{example}

\subsection{Cocycle superrigidity for actions of amalgamated free products}

Recall from paragraph \ref{subsec.cocycle} the notion of cocycle superrigidity.
Theorem \ref{thm.stable-superrigid} will be proven as a consequence of Theorem \ref{thm.uniqueCartan} and a number of cocycle superrigidity theorems from \cite{P0,P-gap} that we recall here.

We first prove the following permanence lemma for $\cL$-cocycle superrigidity. It is a direct consequence of techniques in \cite[Section 3]{P0}, but  we give a short and full proof for the convenience of the reader. We also deal with arbitrary finite index subgroups, which will be useful in Section \ref{sec.bimodules}.

\begin{lemma}\label{lemma.permanence}
Let $\Gamma = \Gamma_1 *_\Sigma \Gamma_2$ be an amalgamated free product where $\Sigma \lhd \Gamma_2$ is a normal subgroup. Let $\Gamma \actson (X,\mu)$ be a p.m.p.\ action and assume that its restriction $\Sigma \actson (X,\mu)$ is weakly mixing. Let $\cL$ be a Polish group whose topology is induced by a separable complete bi-invariant metric (for instance, $\cL \in \Ufin$).

If for every finite index subgroup $\Gamma_1' < \Gamma_1$ the action $\Gamma_1' \actson (X,\mu)$ is $\cL$-cocycle superrigid, then $\Gamma' \actson (X,\mu)$ is $\cL$-cocycle superrigid for all finite index subgroups $\Gamma' < \Gamma$.
\end{lemma}
\begin{proof}
From \cite[Section 3]{P0} we deduce the following two principles. Let $\cL$ be as in the lemma, $\Lambda \actson (X,\mu)$ a p.m.p.\ action and $\Lambda_0 < \Lambda$ a subgroup such that $\Lambda_0 \actson (X,\mu)$ is weakly mixing. Let $\om : \Lambda \times X \recht \cL$ be a $1$-cocycle such that $\om(h,x) = \delta(h)$ for all $h \in \Lambda_0$, where $\delta : \Lambda_0 \recht \cL$ is a group morphism.

Principle 1. If $\Lambda_0$ is normal in $\Lambda$, then $\om$ is a group morphism on the whole of $\Lambda$. Indeed, choose $g \in \Lambda$ and put $\psi(x) = \om(g,x)$. We have to prove that $\psi$ is essentially constant. But $\psi(h \cdot x) = \delta(ghg^{-1}) \, \psi(x) \, \delta(h)^{-1}$ for all $h \in \Lambda_0$ and a.e.\ $x \in X$. By \cite[Lemma 5.4]{PV4} the function $\psi$ is essentially constant.

Principle 2. If $\om$ is cohomologous to a group morphism on $\Lambda$, then $\om$ is already a group morphism on $\Lambda$. Indeed, by assumption we find a measurable map $\psi : X \recht \cL$ and a group morphism $\rho : \Lambda \recht \cL$ such that $\om(g,x) = \psi(g \cdot x)^{-1} \rho(g) \psi(x)$ for all $g \in \Lambda$ and a.e.\ $x \in X$. It suffices to prove that $\psi$ is essentially constant. But $\psi(h \cdot x) = \rho(h) \psi(x) \delta(h)^{-1}$ for all $h \in \Lambda_0$ and a.e.\ $x \in X$. Again by \cite[Lemma 5.4]{PV4} the function $\psi$ is essentially constant.

Take a finite index subgroup $\Gamma' < \Gamma$ and a $1$-cocycle $\om : \Gamma' \times X \recht \cL$. We have to prove that $\om$ is cohomologous to a group morphism. Since $\Gamma_1 \cap \Gamma'$ is a finite index subgroup of $\Gamma_1$, we may assume that $\om$ is already a group morphism on $\Gamma_1 \cap \Gamma'$. We prove that $\om$ is then a group morphism on the whole of $\Gamma'$.

Take a finite index subgroup $\Gamma\dpr < \Gamma'$ such that $\Gamma\dpr$ is normal in $\Gamma$. Define $\Gamma_i' = \Gamma_i \cap \Gamma\dpr$ and $\Sigma' = \Sigma \cap \Gamma\dpr$. Note that $\Sigma' < \Sigma$ has finite index and that $\Gamma_i' \lhd \Gamma_i$, $\Sigma' \lhd \Gamma_2'$ are normal subgroups. Also $\Sigma' \actson (X,\mu)$ is still weakly mixing.

We prove by induction on $|g|$ that $\om$ is a group morphism on $g \Gamma_i' g^{-1}$ for all $g \in \Gamma$ and $i = 1,2$. For $|g|=0$ we already know that $\om$ is a group morphism on $\Gamma_1'$. In particular, $\om$ is a group morphism on $\Sigma'$ and by principle~1, also on $\Gamma_2'$. Assume that the statement is true for all elements of length $n-1$. Take $g \in \Gamma$ with $|g|=n$. Write $g = g_0 h$ where $|g_0| = n-1$ and $h \in \Gamma_1 \cup \Gamma_2$. If $h \in \Gamma_1$, we have $g \Gamma_1' g^{-1} = g_0 \Gamma_1' g_0^{-1}$. By the induction hypothesis, $\om$ is a group morphism on $g \Gamma_1' g^{-1}$ and, in particular, on $g \Sigma' g^{-1}$. By principle~1, $\om$ is also a group morphism on $g \Gamma_2' g^{-1}$. Next, if $h \in \Gamma_2$, we have $g \Gamma_2' g^{-1} = g_0 \Gamma_2' g_0^{-1}$. By the induction hypothesis, $\om$ is a group morphism on $g \Gamma_2' g^{-1}$ and, in particular, on $g \Sigma' g^{-1}$. Consider the $1$-cocycle $\mu : \Gamma_1' \times X \recht \cL$ given by $\mu(h,x) = \om(g h g^{-1}, g \cdot x)$. Since $\Gamma_1' \actson (X,\mu)$ is $\cL$-cocycle superrigid, it follows that $\mu$ is cohomologous to a group morphism. But $\mu$ is already a group morphism on $\Sigma'$. By principle~2, $\mu$ is a group morphism on $\Gamma_1'$ and hence, $\om$ is a group morphism on $g \Gamma_1' g^{-1}$.

Define the subgroup $\Gamma\tpr < \Gamma'$ generated by $g \Gamma_i' g^{-1}$ for all $g \in \Gamma$ and $i = 1,2$. We have already proven that $\om$ is a group morphism on $\Gamma\tpr$. By construction, $\Gamma\tpr$ is normal in $\Gamma'$. So, by principle~1, $\om$ is a group morphism on $\Gamma'$.
\end{proof}

\begin{theorem}\label{thm.cocyclesuperrigid}
Let $\Gamma = \Gamma_1 *_\Sigma \Gamma_2$ be an amalgamated free product over an infinite subgroup $\Sigma$ that is normal in $\Gamma_2$. Let $\Gamma \actson (X,\mu)$ be a free ergodic p.m.p.\ action. Assume that the group $\Gamma$ and its action $\Gamma \actson (X,\mu)$ satisfy condition 1 or condition 2 in Theorem \ref{thm.stable-superrigid}. Then, for every finite index subgroup $\Gamma' < \Gamma$, the action $\Gamma' \actson (X,\mu)$ is $\Ufin$-cocycle superrigid.
\end{theorem}
\begin{proof}
Both conditions 1 and 2 imply that $\Sigma \actson (X,\mu)$ is weakly mixing. By Lemma \ref{lemma.permanence}, it suffices to prove that $\Gamma_1' \actson (X,\mu)$ is $\Ufin$-cocycle superrigid for every finite index subgroup $\Gamma_1' < \Gamma_1$. Under conditions 1, this is a consequence of \cite[Theorem 0.1]{P0}. Under conditions 2, this follows from \cite[Theorem 1.1]{P-gap}.
\end{proof}

Finally, for later use in Section \ref{sec.bimodules} we record the following lemma about weak $1$-cocycles (cf.\ \cite[Theorem 4.1]{PSa}).

\begin{lemma}\label{lemma.weakcocycle}
Let $\Gamma \actson (X,\mu)$ be a p.m.p.\ action that is $\Ufin$-cocycle superrigid. Assume that $\Om : \Gamma \times \Gamma \recht S^1$ is a scalar $2$-cocycle on $\Gamma$ and that $\om : \Gamma \times X \recht S^1$ is a measurable map satisfying
$$\om(gh,x) = \Om(g,h) \om(g,h \cdot x) \om(h,x)$$
for all $g,h \in \Gamma$ and a.e.\ $x \in X$. Then, there exist a measurable function $\vphi : X \recht S^1$ and a map $\delta : \Gamma \recht S^1$ such that
$$\om(g,x) = \vphi(g \cdot x) \delta(g) \overline{\vphi(x)}$$
for all $g \in \Gamma$ and a.e.\ $x \in X$. In particular, $\Om$ is a coboundary: $\Om(g,h) = \delta(gh) \overline{\delta(g)} \overline{\delta(h)}$ for all $g,h \in \Gamma$.
\end{lemma}
\begin{proof}
Form the twisted group von Neumann algebra $N := \rL_\Omega(\Gamma)$ generated by unitaries $u_g$, $g \in \Gamma$ satisfying $u_g u_h = \Om(g,h) u_{gh}$ and equipped with a trace $\tau$ satisfying $\tau(u_g) = 0$ if $g \neq e$. Denote by $\cL$ the closed subgroup of $\cU(N)$ consisting of the unitaries $\lambda u_g$ with $\lambda \in S^1$, $g \in \Gamma$. Writing $\mu(g,x) = \om(g,x) u_g$, it follows that $\mu$ is a $1$-cocycle for $\Gamma \actson (X,\mu)$ with values in $\cL$.

Since $\Gamma \actson (X,\mu)$ is $\Ufin$-cocycle superrigid, we find a measurable map $\theta : X \recht \cL$ and a group morphism $\eta : \Gamma \recht \cL$ such that $\mu(g,x) = \theta(g \cdot x) \eta(g) \theta(x)^{-1}$ for all $g \in \Gamma$ and a.e.\ $x \in X$. Denoting by $\vphi(x)$, resp.\ $\delta(g)$, the $S^1$-part of $\theta(x)$, resp.\ $\eta(g)$, the lemma is proven.
\end{proof}

\subsection{Proof of Theorem \ref{thm.stable-superrigid}}

\begin{proof}[Proof of Theorem \ref{thm.stable-superrigid}]
Since $\Gamma$ belongs to $\cG$ the conclusion of Theorem \ref{thm.uniqueCartan} holds. By Theorem \ref{thm.cocyclesuperrigid}, $\Gamma \actson (X,\mu)$ is $\Ufin$-cocycle superrigid. In particular, $\Gamma \actson (X,\mu)$ is cocycle superrigid with arbitrary countable target groups and with target group $S^1$. So, the conclusions follow from Lemma \ref{lemma.assemble}.
\end{proof}

\subsection{Strong rigidity for group von Neumann algebras}

Following up Connes' rigidity conjecture \cite{C80bis}, Jones asked in \cite{Jo00}
whether every isomorphism between property (T) group von Neumann algebras $\rL(G_1)$
and $\rL(G_2)$ essentially comes from an isomorphism between the groups $G_1$
and $G_2$ (cf.\ \cite[Statement 3.2']{P-ICM}). In \cite[Theorem 7.13]{PV1},
we provided a family of (generalized) wreath product groups satisfying such a
strong rigidity result. The following corollary enlarges this family of groups.

\begin{corollary}\label{cor.strong-rigidity-groups}
Consider all group actions $\Gamma \actson I$ on countable sets covered by Theorem \ref{thm.stable-superrigid}. Assume moreover that $\Gamma$ has no finite normal subgroups and that $(\Stab i) \cdot j$ is infinite for all $i,j \in I$ with $i \neq j$. Denote by $\cM$ the family of associated wreath product groups $G = \bigl( \Z / 2\Z \bigr)^{(I)} \rtimes \Gamma$.

If $G_1,G_2 \in \cM$, $t > 0$ and if $\pi : \rL(G_1) \recht \rL(G_2)^t$ is a stable isomorphism between $\rL(G_1)$ and $\rL(G_2)$, then the groups $G_1$, $G_2$ are isomorphic through an isomorphism $\delta : G_1 \recht G_2$, the amplification $t$ equals $1$ and there exist a unitary $w \in \rL(G_2)$ and a group morphism $\om : G_1 \recht S^1$ such that $$\pi(u_g) = \om(g) \; w \, v_{\delta(g)} w^* \quad\text{for all}\;\; g \in G_1 \; .$$
Here, $(u_g)_{g \in G_1}$ and $(v_s)_{s \in G_2}$ denote the natural unitaries generating $\rL(G_1)$ and $\rL(G_2)$.

In particular, the automorphism group of $\rL(G)$ for $G \in \cM$, is generated by $\Aut(G)$, the group of characters $\Char G$ and the inner automorphisms.
\end{corollary}

To prove Corollary \ref{cor.strong-rigidity-groups}, we use the following lemma, whose proof is contained in \cite[Theorem 5.4]{PV1}. Observe that Lemma \ref{lemma.conjugacy} implies in particular that, whenever $\Gamma \actson I$ is such that $(\Stab j) \cdot k$ is infinite for all $j \neq k$, the generalized Bernoulli actions $\Gamma \actson (X_0,\mu_0)^I$ and $\Gamma \actson (Y_0,\rho_0)^I$ are conjugate if and only if the base probability spaces are isomorphic. Such a statement is false in general for plain Bernoulli actions $\Gamma \actson (X_0,\mu_0)^\Gamma$, see \cite{stepin}.

\begin{lemma} \label{lemma.conjugacy}
Let $\Gamma_i \actson I_i$, $i = 1,2$. Assume that for both actions, $(\Stab j) \cdot k$ is infinite for all $j \neq k$. Consider the generalized Bernoulli actions $$\Gamma_1 \actson (X,\mu) = (X_0,\mu_0)^{I_1} \quad\text{and}\quad \Gamma_2 \actson (Y,\rho) := (Y_0,\rho_0)^{I_2} \; .$$
If $\Delta : X \recht Y$ is an isomorphism of probability spaces and $\delta : \Gamma_1 \recht \Gamma_2$ an isomorphism of groups such that $\Delta(g \cdot x) = \delta(g) \cdot \Delta(x)$ for all $g \in \Gamma_1$ and almost all $x \in X$, there exists a bijection $\eta : I_1 \recht I_2$ and there exist isomorphisms $\Delta_i : X_0 \recht Y_0$ of the base probability spaces such that
\begin{itemize}
\item $\eta(g \cdot i) = \delta(g) \cdot \eta(i)$ and $\Delta_{g \cdot i} = \Delta_i$ for all $i \in I_1$, $g \in \Gamma_1$,
\item $\bigl(\Delta(x)\bigr)_{\eta(i)} = \Delta_i(x_i)$ for all $i \in I_1$ and almost all $x \in X$.
\end{itemize}
\end{lemma}

\begin{proof}[Proof of Corollary \ref{cor.strong-rigidity-groups}]
Write $G_i = \bigl(\Z / 2\Z\bigr)^{(I_i)} \rtimes \Lambda_i$ and identify $\rL(G_i) = A_i \rtimes \Lambda_i$, where $A_i = \rL\Bigl(\bigl(\Z / 2\Z\bigr)^{(I_i)}\Bigr)$. Let $\pi : \rL(G_1) \recht \rL(G_2)^t$ be a $*$-isomorphism. By Theorem \ref{thm.stable-superrigid}, we get that $t=1$ and we get the existence of an isomorphism $\delta : G_1 \recht G_2$, a character $\om : \Lambda_1 \recht S^1$ and a unitary $w \in \rL(G_2)$ such that
$$\pi(a u_g) = \om(g) \; w \, \al(a) \, v_{\delta(g)} \, w^* \; ,$$
where $\al : A_1 \recht A_2$ is a $*$-isomorphism satisfying $\al(\si_g(a)) = \si_{\delta(g)}(a)$ for all $g \in \Lambda_1$ and $a \in A_1$.

Lemma \ref{lemma.conjugacy} describes the form of $\pi|_{A_1}$. Since $\rL(\Z/2\Z)$ has precisely two automorphisms, the identity and the multiplication with the non-trivial character on $\Z/2\Z$, we are done.
\end{proof}

\subsection{Counterexamples to W$^*$-superrigidity}

There are many free ergodic p.m.p.\ actions $\Gamma \actson (X,\mu)$ with $\Gamma \in \cG$ and such that $\Gamma \actson (X,\mu)$ is orbit equivalent with a large family of non-conjugate group actions. So, in these cases, the II$_1$ factor $\rL^\infty(X) \rtimes \Gamma$ has many group measure space decompositions (up to conjugacy of the actions), but all of them have the same Cartan subalgebra (up to unitary conjugacy).

First of all, by \cite[P$_\text{\rm ME}$6]{G3}, if $\Gamma_i$ and $\Gamma_i'$ admit orbit equivalent actions ($i=1,2$), then $\Gamma_1 * \Gamma_2$ and $\Gamma_1' * \Gamma_2'$ also admit orbit equivalent actions.

Next, let $\Gamma_1$ be an arbitrary group and $\Gamma_2,\Gamma_2'$ infinite amenable groups. Denote $\Gamma = \Gamma_1 * \Gamma_2$ and $\Gamma' = \Gamma_1 * \Gamma_2'$. By \cite{bowen}, the Bernoulli actions $\Gamma \actson (X_0,\mu_0)^\Gamma$ and $\Gamma' \actson (Y_0,\eta_0)^{\Gamma'}$ are orbit equivalent for all non-trivial base probability spaces $(X_0,\mu_0)$ and $(Y_0,\eta_0)$.

\section{W$^*$-superrigidity for finite index bimodules} \label{sec.bimodules}

By definition, the action $\Gamma \actson (X,\mu)$ is stably W$^*$-superrigid if,
for all free ergodic p.m.p.\ actions $\Lambda \actson (Y,\eta)$,
stable isomorphism of $\rL^\infty(X) \rtimes \Gamma$ and
$\rL^\infty(Y) \rtimes \Lambda$ implies, in a sense,
stable isomorphism of the groups $\Gamma$, $\Lambda$ and their respective actions.

For certain of the actions listed in Theorem \ref{thm.stable-superrigid},
one can go even further and prove that, whenever $\cH$ is a finite index bimodule between $\rL^\infty(X) \rtimes \Gamma$ and $\rL^\infty(Y) \rtimes \Lambda$, the groups $\Gamma$, $\Lambda$ and their actions are, in the following precise sense, virtually isomorphic, with $\cH$ being implemented by this virtual conjugacy and a finite dimensional unitary representation.

However, the existence of $g_1,\ldots,g_n \in \Gamma$ with $\bigcap_{i=1}^n g_i \Sigma g_i^{-1}$ finite, does not provide sufficient absence of normality of $\Sigma < \Gamma$. Therefore, we make in the following theorem, a much stronger (and certainly non-optimal) assumption.

\begin{theorem}\label{thm.bimodules}
Let $\Gamma \actson (X,\mu)$ be any of the actions listed in Theorem \ref{thm.stable-superrigid} and assume that $\Gamma$ admits an infinite index subgroup $G$ such that $g \Sigma g^{-1} \cap \Sigma$ is finite for all $g \in \Gamma - G$. If $\Lambda \actson (Y,\eta)$ is an arbitrary free ergodic p.m.p.\ action and $\cH$ is a finite index $\bigl(\rL^\infty(X) \rtimes \Gamma\bigr)$--$\bigl(\rL^\infty(Y) \rtimes \Lambda\bigr)$--bimodule, there exist
\begin{itemize}
\item a finite index subgroup $\Gamma' < \Gamma$,
\item a finite subgroup $H < \Aut(X,\mu)$ such that $g H g^{-1} = H$ for all $g \in \Gamma'$,
\item finite index subgroups $\Lambda\dpr < \Lambda' < \Lambda$ with $\Lambda \actson Y$ being induced from $\Lambda' \actson Y'$ for some $Y' \subset Y$,
\end{itemize}
such that $\Lambda\dpr \cong \frac{\Gamma'}{\Gamma' \cap H}$ and such that the action $\Lambda\dpr \actson Y'$ is conjugate with $\frac{\Gamma'}{\Gamma' \cap H} \actson \frac{X}{H}$, through an isomorphism $\Delta : \frac{X}{H} \recht Y'$ of probability spaces and a group isomorphism $\delta : \frac{\Gamma'}{\Gamma' \cap H} \recht \Lambda\dpr$.

Moreover, if $\cH$ is irreducible, there exists an irreducible finite dimensional unitary representation $\pi : \Gamma' \recht \cU(\C^k)$ such that $\cH$ is the composition of
\begin{itemize}
\item the \mybim{\rL^\infty(X) \rtimes \Gamma}{\rL^\infty(X) \rtimes \Gamma'} on the Hilbert space $\M_{1,k}(\C) \ot \rL^2(M)$, where $M = \rL^\infty(X) \rtimes \Gamma$ and where
\begin{equation}\label{eq.firstbim}
z \cdot \xi \cdot (a u_g) = z \xi (\pi(g) \ot a u_g) \quad\text{for all}\;\; z \in M, a \in \rL^\infty(X), g \in \Gamma' \; ,
\end{equation}
\item the natural \mybim{\rL^\infty(X) \rtimes \Gamma'}{\rL^\infty(\frac{X}{H}) \rtimes \frac{\Gamma'}{\Gamma' \cap H}} on the Hilbert space $\rL^2(M_0) p_{\Gamma' \cap H}$, where $$M_0 = \rL^\infty(X) \rtimes \Gamma' \quad\text{and}\quad p_{\Gamma' \cap H} = |\Gamma' \cap H|^{-1} \sum_{g \in \Gamma' \cap H} u_g \; , $$
\item the \mybim{\rL^\infty(\frac{X}{H}) \rtimes \frac{\Gamma'}{\Gamma' \cap H}}{\rL^\infty(Y) \rtimes \Lambda} on the Hilbert space $\rL^2(N)$, where $N = \rL^\infty(Y) \rtimes \Lambda$ and where
$$(a u_g) \cdot \xi \cdot z = (a \circ \Delta^{-1}) u_{\delta(g)} \xi z \quad\text{for all}\;\; a \in \rL^\infty({\textstyle \frac{X}{H}}), g \in {\textstyle \frac{\Gamma'}{\Gamma' \cap H}} , z \in N \; .$$
\end{itemize}
\end{theorem}

\begin{example} \label{ex.bimodules}
Let $\Sigma$ be an amenable group and $\Gamma = \Gamma_1 *_\Sigma \Gamma_2$. Assume that $\Gamma \actson I$ with $\Sigma \cdot i$ being infinite for every $i \in I$. If one of the following conditions is satisfied, the generalized Bernoulli action $\Gamma \actson (X_0,\mu_0)^I$ satisfies the conclusions of Theorem \ref{thm.bimodules}.
\begin{enumerate}
\item Let $\Sigma = \Z$ and embed $\Sigma$ into $\Gamma_1 = \SL(n,\Z)$, $n \geq 3$, in the upper right corner. Assume that $\Sigma \lhd \Gamma_2$ is a proper normal subgroup. Then, Theorem \ref{thm.bimodules} applies by putting $G = G_0 *_\Sigma \Gamma_2$, where $G_0$ consists of those matrices $A$ with $A_{i1} = 0$ for all $2 \leq i \leq n$.
\item Let $\Sigma$ be a common subgroup of $\Lambda_0,\Lambda_1,\Gamma_2$ and assume that $\Sigma$ is normal in $\Gamma_2$, $\Sigma \neq \Gamma_2$ and $\Lambda_0$ is non-amenable. Let $\Lambda_2$ be an arbitrary non-trivial group. Put $\Gamma_1 = \Lambda_0 \times (\Lambda_1 * \Lambda_2)$ and view $\Sigma < \Gamma_1$ diagonally. Theorem \ref{thm.bimodules} applies by putting $G = (\Lambda_0 \times \Lambda_1) *_\Sigma \Gamma_2$. Assume that $\Lambda_0$ or $\Lambda_1 * \Lambda_2$ acts with infinite orbits on $I$.
\end{enumerate}
\end{example}

In the proof of Theorem \ref{thm.bimodules} we make use of the following lemma having some independent interest.

\begin{lemma}\label{lemma.finitequotient}
Let $\Gamma \actson (X,\mu)$ and $\Lambda \actson (Y,\eta)$ be free p.m.p.\ actions. Assume that $\Gamma \actson (X,\mu)$ is cocycle superrigid with finite target groups and that all finite index subgroups of $\Gamma$ act ergodically on $(X,\mu)$. Let $\Delta : X \recht Y$ be an $m$-to-$1$ measure preserving map satisfying $\Delta(g \cdot x) = \delta(g) \cdot x$ for some surjective group morphism $\delta : \Gamma \recht \Lambda$.

Then there exists a group $H$ with $m$ elements and a free p.m.p.\ action $H \actson (X,\mu)$ such that, viewing both $\Gamma$ and $H$ as subgroups of $\Aut(X,\mu)$, we have
\begin{itemize}
\item $g H g^{-1} = H$ for all $g \in \Gamma$,
\item $\Ker \delta = \Gamma \cap H$,
\item $\Delta$ is the composition of the canonical quotient map $X \recht \frac{X}{H}$ with a conjugacy of the actions $\frac{\Gamma}{\Gamma \cap H} \actson \frac{X}{H}$ and $\Lambda \actson Y$.
\end{itemize}
\end{lemma}
\begin{proof}
Partition $X$ into disjoint measurable subsets $X_1,\ldots,X_m$ such that the restrictions $\Delta_i := \Delta_{|X_i}$ are isomorphisms of $X_i$ onto $Y$ scaling the measure by the factor $1/m$.

Put $\cJ := \{1,\ldots,m\}$. For $g \in \Gamma$ and almost every $x \in X$, define the map $\om(g,x) : \cJ \recht \cJ$ such that
$$\om(g,x) i = j \quad\text{if and only if}\quad g \cdot \Delta_i^{-1}(\Delta(x)) \in X_j \; .$$
So, by construction
$$g \cdot \Delta_i^{-1}(\Delta(x)) = \Delta_{\om(g,x)i}^{-1}(\Delta(g \cdot x))$$
almost everywhere. It follows that $\om(gh,x) = \om(g,h\cdot x) \circ \om(h,x)$ almost everywhere.
Hence, $\om(g,x)$ is a permutation of $\cJ$ and $\om$ is a $1$-cocycle for the action $\Gamma \actson X$ with values in the permutation group $S_m$. Since $\Gamma \actson (X,\mu)$ is assumed to be cocycle superrigid with finite target groups, we find a measurable map $\vphi : X \recht S_m$ and a group morphism $\eta : \Gamma \recht S_m$ such that, writing
$$T_i : X \recht X : T_i(x) = \Delta_{\vphi(x)i}^{-1}(\Delta(x)) \; ,$$
we have $g \cdot T_i(x) = T_{\eta(g)i}(g \cdot x)$ almost everywhere. By construction, every $T_i$ is locally a m.p.\ isomorphism. Moreover, the range of $T_i$ is globally invariant under the finite index subgroup $\Ker \eta < \Gamma$. Hence, all $T_i$ are m.p.\ automorphisms of $(X,\mu)$.

Define the equivalence relation $\cR$ on $X$ as $\cR := \{(x,y) \in X \times X \mid \Delta(x) = \Delta(y)\}$. By construction, every equivalence class has $m$ elements and the graph of every $T_i$ belongs to $\cR$. Finally, if $T_i(x) = T_j(x)$ for all $x$ in a non-negligible subset $\cU$, we can make $\cU$ smaller and assume moreover that $\vphi(x) = \si$ for all $x \in \cU$. Hence, $\Delta_{\si i}^{-1}(\Delta(x)) = \Delta_{\si j}^{-1}(\Delta(x))$ for all $x \in \cU$. So, $\si i = \si j$ and hence, $i = j$. We have shown that $\cR$ is the disjoint union of the graphs of the $T_i$, $i=1,\ldots,m$.

Define $H < \Aut(X,\mu)$ as the group generated by $T_1,\ldots,T_m$. By construction, all $h \in H$ commute with $\Ker \eta$. Since $\Ker \eta$ acts ergodically on $(X,\mu)$, it follows that $H$ acts freely on $(X,\mu)$. Since the orbit equivalence relation of $H \actson X$ is contained in $\cR$, we get that $|H| \leq m$. But $H$ contains the distinct elements $T_1,\ldots,T_m$, so that $H = \{T_1,\ldots,T_m\}$.

It remains to prove that $\Ker \delta = \Gamma \cap H$. Since $\Delta(g \cdot x) = \delta(g) \cdot \Delta(x)$ and $\Lambda$ acts freely on $Y$, it follows that an element $g \in \Gamma$ belongs to $\Ker \delta$ if and only if the graph of $g$ belongs to $\cR$. It follows that $|\Ker \delta| \leq m$ and that $\Gamma \cap H \subset \Ker \delta$. Conversely, if $g \in \Ker \delta$, take $i$ such that the graphs of $g$ and $T_i$ intersect non-negligibly. Since $\Ker \delta$ is a finite normal subgroup of $\Gamma$, it follows that $g$ commutes with a finite index subgroup of $\Gamma$. We have already seen that also $T_i$ commutes with a finite index subgroup of $\Gamma$. It follows that $g$ and $T_i$ coincide almost everywhere, i.e.\ $g \in \Gamma \cap H$.
\end{proof}

\begin{proof}[Proof of Theorem \ref{thm.bimodules}]
As a preliminary step, note that Lemmas \ref{lemma.crucial} and \ref{lemma.spectralgap} remain valid when we identify $Q \rtimes \Lambda$ with a finite index subfactor of $pMp$. This generalization almost has the same proof and we explain this briefly for the case of Lemma \ref{lemma.crucial}. The deformation $\vphi_n$ of $M$ still allows to define the completely positive maps $\vphitil_n$ on $\Lambda$, which in turn lead to a new deformation $\theta_n$ of $Q \rtimes \Lambda$. By Jones' tunnel construction and for the appropriate value of $s > 0$, we can view the amplification $M^s$ as a finite index subfactor of $Q \rtimes \Lambda$. Hence, we find inside $Q \rtimes \Lambda$ a von Neumann subalgebra $M_0$ with the relative property (T) and without injective direct summand. So, we can finish the proof in the same way as for the original Lemma \ref{lemma.crucial}.

Write $A = \rL^\infty(X)$ and $M = A \rtimes \Gamma$. Put $B = \rL^\infty(Y)$ and $N = B \rtimes \Lambda$. We may assume that $\bim{M}{\cH}{N}$ is an irreducible finite index bimodule. Let $\gamma : N \recht pM^m p$ be a finite index inclusion such that
$$\bim{M}{\cH}{N} \cong \bim{M}{\bigl( (\M_{1,m}(\C) \ot \rL^2(M))p \bigr)}{\gamma(N)} \; .$$
View $M^m$ as the crossed product $(\M_m(\C) \ot A) \rtimes \Gamma$ where $\Gamma$ acts trivially on $\M_m(\C)$.
As in Section \ref{sec.deformation-AFP}, define for every $0 < \rho < 1$, the unital completely positive map $\m_\rho$ on $M^m$ by $\m_\rho(a u_g) = \rho^{|g|} a u_g$ for all $a \in \M_m(\C) \ot A$, $g \in \Gamma$.

In the situation where $\Gamma_1$ has a non-amenable subgroup with the relative property (T), apply Lemma \ref{lemma.crucial} -- as generalized in the first paragraph of the proof -- to the completely positive maps $\m_\rho$, $\rho \recht 1$ and the von Neumann subalgebra $P = (A \rtimes \Sigma)^m$. In the situation where $\Gamma_1$ has two non-amenable commuting subgroups $H_1,H_2$, we apply the generalized Lemma \ref{lemma.spectralgap}. Put $\eps = (\operatorname{tr} \ot \tau)(p)/2072$. Combined with \eqref{eq.some-estimate}, we always find $0 < \rho < 1$ and a sequence $s_k \in \Lambda$ such that
\begin{itemize}
\item $\|\gamma(v_{s_k}) - \m_\rho(\gamma(v_{s_k}))\|_2 \leq \eps/2$ for all $k$.
\item $\|E_P(x \gamma(v_{s_k}) y)\|_2 \recht 0$ for all $x,y \in M$.
\end{itemize}
By Lemma \ref{lem.lengths} we get a $0 < \rho < 1$ and a $\delta > 0$ such that $\tau(\gamma(w)^* \m_\rho(\gamma(w))) \geq \delta$ for all $w \in \cU(B)$. By Theorem \ref{thm.IPP} we have that $\gamma(B) \embed{M} A \rtimes \Sigma$ or that $\cN_{pM^m p}(\gamma(B))\dpr \embed{M} A \rtimes \Gamma_i$ for some $i=1,2$. But $\cN_{pM^m p}(\gamma(B))$ contains $\gamma(N)$ which has finite index in $pM^m p$. Since $\Gamma_i < \Gamma$ has infinite index, this rules out the possibility that $\cN_{pM^m p}(\gamma(B))\dpr \embed{M} A \rtimes \Gamma_i$. So, we have shown that $\gamma(B) \embed{M} A \rtimes \Sigma$.

We claim that $\gamma(B) \embed{M} A$. Assume the contrary. Our assumption on $G$, together with the regularity of $B \subset N$, implies that $\gamma(N) \embed{M} A \rtimes G$ (cf.\ \cite[Lemma 4.2]{V-bim}). Since $\gamma(N) \subset pM^m p$ is of finite index, while $G < \Gamma$ has infinite index, this is a contradiction. So, we have shown that $\gamma(B) \embed{M} A$.

In bimodule language, this means that $\cH$ admits a non-zero $A$-$B$-subbimodule $\cK$ which is finitely generated as a left $A$-module. Take a finite index inclusion $\psi : M \recht q N^k q$ such that
$$\bim{M}{\cH}{N} \cong \bim{\psi(M)}{q (\C^k \ot \rL^2(N))}{N} \; .$$
The existence of $\cK$ means that $B \embed{N} \psi(A)$. Taking relative commutants (see \cite[Lemma 3.5]{V-bim}), it follows that $\psi(A) \embed{N} B$.

By Theorem \ref{thm.cocyclesuperrigid} the action $\Gamma' \actson X$ is $\Ufin$-cocycle superrigid for every finite index subgroup $\Gamma' < \Gamma$. Therefore, we are in a situation where we can apply \cite[Lemma 6.5]{V-bim} as well as the first part of the proof of \cite[Theorem 6.4]{V-bim}. Denote by $\D_k(\C) \subset \M_k(\C)$ the algebra of diagonal matrices. As a result, we get
\begin{itemize}
\item a finite index subgroup $\Gamma' < \Gamma$ and an irreducible projective representation $\pi : \Gamma' \recht \cU(\C^k)$ with obstruction $2$-cocycle $\Omega : \Gamma' \times \Gamma' \recht S^1$ given by $\pi(g) \pi(h) = \Omega(g,h) \pi(gh)$,
\item a projection $q' \in \D_k(\C) \ot B$ and a finite index inclusion $\psi' : A \rtimes_\Omega \Gamma' \recht q' N^k q'$ satisfying $\psi'(A) = (\D_k(\C) \ot B)q'$,
\end{itemize}
such that $\cH$ is the composition of the \mybim{A \rtimes \Gamma}{A \rtimes_\Omega \Gamma'} given by \eqref{eq.firstbim} and the $\bigl(A \rtimes_\Omega \Gamma'\bigr)$--$N$--bimodule given by
$$\bim{\psi'\bigl(A \rtimes_\Omega \Gamma'\bigr)}{q'(\C^k \ot \rL^2(N))}{N}\; .$$
Denote $\Lambdatil = \frac{\Z}{k\Z} \times \Lambda$ and $\Ytil = \frac{\Z}{k\Z} \times Y$. Consider the natural action $\Lambdatil \actson \Ytil$ and define $Z \subset \Ytil$ such that $q' = \chi_{Z}$. Denote by $T : X \recht Z$ the isomorphism of measure spaces determined by $\psi'(a) = a \circ T^{-1}$ for all $a \in A$. Normalize the measure $\eta$ on $Y$ in such a way that $Z$ has measure $1$ and hence, $T$ is measure preserving.

By construction, $T(\Gamma' \cdot x) \subset \Lambdatil \cdot T(x)$ for almost all $x \in X$. So, the formula
$$T(g \cdot x) = \om(g,x) \cdot T(x)$$
defines a $1$-cocycle for the action $\Gamma' \actson X$ with values in $\Lambdatil$. By Theorem \ref{thm.cocyclesuperrigid}, $\Gamma' \actson X$ is $\Ufin$-cocycle superrigid. So, we find a measurable map $\vphi : X \recht \Lambdatil$ and a group morphism $\delta : \Gamma' \recht \Lambdatil$ such that, writing $\Delta(x) := \vphi(x) \cdot T(x)$, we have $\Delta(g \cdot x) = \delta(g) \cdot \Delta(x)$ almost everywhere. Since $\Gamma' \actson X$ is weakly mixing, we find $i \in \frac{\Z}{k\Z}$ such that $\Delta(x) \in \{i\} \times Y$ for almost all $x \in X$. Hence, $\delta(\Gamma') \subset \{1\} \times \Lambda$. From now on, we view $\Delta$ as a map from $X$ to $Y$ and $\delta$ as a group morphism from $\Gamma'$ to $\Lambda$. Put $\Lambda\dpr := \delta(\Gamma')$.

By construction, $\Delta$ is locally a m.p.\ isomorphism, meaning that we can partition $X$ into a sequence of measurable subsets $X_n$ such that the restriction of $\Delta$ to each of the $X_n$ is a m.p.\ isomorphism of $X_n$ onto a subset of $Y$. Since $\Delta(g \cdot x) = \delta(g) \cdot \Delta(x)$ and since $(X,\mu)$ is a probability space, it follows that $\Ker \delta$ is finite and that $\Delta$ is an $m$-to-$1$ quotient map onto $Y' \subset Y$ (cf.\ \cite[Theorem 1.8]{furman-on-popa}). Observe that $Y'$ is globally $\Lambda\dpr$-invariant. Partition $X$ into subsets $X_1,\ldots,X_m$ of measure $1/m$ such that $\Delta_i := \Delta_{|X_i}$ is an isomorphism of $X_i$ onto $Y'$.

Define the projection $p_{\Ker \delta}$ in $A \rtimes \Gamma'$ by the formula
$$p_{\Ker \delta} := |\Ker \delta|^{-1} \sum_{g \in \Ker \delta} u_g \; .$$
Put $B_1 := \rL^\infty(Y') = \rL^\infty(Y) q_1$, where $q_1 : = \chi_{Y'}$.
The pair $\Delta,\delta$ yields a natural \mybim{A \rtimes \Gamma'}{B_1 \rtimes \Lambda\dpr} structure on the Hilbert space $\cK := \rL^2(A \rtimes \Gamma')p_{\Ker \delta}$. We can also write this bimodule as
$$\bim{\gamma(A \rtimes \Gamma')}{\bigl(\C^m \ot \rL^2(B_1 \rtimes \Lambda\dpr)\bigr)}{B_1 \rtimes \Lambda\dpr}$$
where $\gamma : A \rtimes \Gamma' \recht \M_m(\C) \ot (B_1 \rtimes \Lambda\dpr)$ is a finite index inclusion satisfying
$$\gamma(a) = \sum_{i=1}^m e_{ii} \ot (a_{|X_i} \circ \Delta_i^{-1}) \quad\text{for all}\;\; a \in A \; .$$
Since $T(x) \in \Lambdatil \cdot \Delta(x)$ for almost all $x \in X$, we can take $W \in \M_{m,k}(\C) \ot N$ satisfying $W W^* = 1 \ot q_1$, $W^* W = q'$ and $\gamma(a) = W \psi'(a) W^*$ for all $a \in A$. Replace $\psi'$ by $W \psi'(\,\cdot\,)W$.

It follows that, for all $g \in \Gamma'$, $\psi'(u_g) \gamma(u_g)^*$ commutes with $\gamma(A) = \psi'(A) = \D_m(\C) \ot B_1$, which is maximal abelian in $\M_m(\C) \ot q_1 N q_1$. So, $\psi'(u_g) = \gamma(\om_g u_g)$ where $\om_g \in \cU(A)$ satisfies
$$\om_g \si_g(\om_h) = \Omega(g,h) \om_{gh} \; .$$
By Lemma \ref{lemma.weakcocycle} we find $w \in \cU(A)$ and a map $\mu : \Gamma' \recht S^1$ such that $\om_g = \mu(g) w \si_g(w^*)$ for all $g \in \Gamma'$. Replacing $g \mapsto \pi(g)$ by $g \mapsto \overline{\mu(g)} \pi(g)$ and replacing $\psi'$ by $(\Ad \psi'(w)^*) \circ \psi'$, we may assume that $\pi$ is an ordinary unitary representation (i.e.\ $\Omega = 1$) and that $\psi'(z) = \gamma(z)$ for all $z \in A \rtimes \Gamma'$. Since $\psi'(A \rtimes \Gamma')$ has finite index in $\M_m(\C) \ot q_1 N q_1$ and since $\gamma(A \rtimes \Gamma')$ is contained in $(B \rtimes \Lambda\dpr)^m$, it follows that $\Lambda\dpr < \Lambda$ is a finite index subgroup.

If $s \in \Lambda$, then $\Delta^{-1}(s \cdot Y' \cap Y')$ is globally invariant under the finite index subgroup $\delta^{-1}(\Lambda\dpr \cap s \Lambda\dpr s^{-1})$ of $\Gamma'$ and hence, must have measure $0$ or $1$. So, either $s \cdot Y' \cap Y'$ has measure zero, or $s \cdot Y'$ equals $Y'$ up to measure zero. Define the subgroup $\Lambda' < \Lambda$ consisting of those $s \in \Lambda$ for which $s \cdot Y'$ equals $Y'$ up to measure zero. By construction, $\Lambda\dpr < \Lambda' < \Lambda$ and $\Lambda \actson Y$ is induced from $\Lambda' \actson Y'$.

The theorem now follows by applying Lemma \ref{lemma.finitequotient} to the actions $\Gamma' \actson (X,\mu)$ and $\Lambda\dpr \actson Y'$.
\end{proof}

\end{document}